\documentclass[final]{siamart1116}

\usepackage{lipsum}
\usepackage{amsfonts}
\usepackage{graphicx}
\usepackage{epstopdf}
\usepackage[noend]{algorithmic}
\usepackage{bm}
\usepackage{fourier}
\usepackage{mathtools}
\usepackage[colorinlistoftodos]{todonotes}
\usepackage{changes} 
\definechangesauthor[name={Elvar}, color=red]{ekb}
\ifpdf
  \DeclareGraphicsExtensions{.eps,.pdf,.png,.jpg}
\else
  \DeclareGraphicsExtensions{.eps}
\fi

\numberwithin{theorem}{section}
\numberwithin{equation}{section}

\newcommand{\TheTitle}{Pass-Efficient Randomized Algorithms for Low-Rank Matrix Approximation Using Any Number of Views} 
\newcommand{\ShortTitle}{Pass-Efficient Randomized Algorithms for Low-Rank Matrix Approximation}  
\newcommand{\TheAuthors}{Elvar K. Bjarkason}

\headers{\ShortTitle}{\TheAuthors}

\title{{\TheTitle}\thanks{Date: May 24, 2018.
\funding{This work was supported by the NZ Ministry of Business, Innovation and Employment through grant C05X1306, as well as an AUEA Braithwaite-Thompson Graduate Research Award and a scholarship from Landsbankinn.}}}

\author{
  Elvar K. Bjarkason\thanks{Department of Engineering Science, University of Auckland, Auckland, New Zealand (\email{ebja558@aucklanduni.ac.nz}).}
}

\usepackage{amsopn}
\DeclareMathOperator{\diag}{diag} 
\DeclareMathOperator{\qr}{qr} 
\DeclareMathOperator{\orth}{orth} 
\DeclareMathOperator{\tsvd}{tsvd} 
\DeclareMathOperator{\tevd}{tevd} 
\DeclareMathOperator{\randn}{randn} 
\DeclareMathOperator{\chol}{chol} 
\DeclareMathOperator{\range}{range} 
\DeclareMathOperator{\QrQc}{constructQrOrQc} 
\newcommand{\norm}[1]{\left\lVert#1\right\rVert} 

\ifpdf
\hypersetup{
  pdftitle={\TheTitle},
  pdfauthor={\TheAuthors}
}
\fi

\begin{document}

\maketitle

\begin{abstract}
This paper describes practical randomized algorithms for low-rank matrix approximation that accommodate any budget for the number of views of the matrix. The presented algorithms, which are aimed at being as pass efficient as needed, expand and improve on popular randomized algorithms targeting efficient low-rank reconstructions. First, a more flexible subspace iteration algorithm is presented that works for any views $v \geq 2$, instead of only allowing an even $v$. Secondly, we propose more general and more accurate single-pass algorithms. In particular, we propose a more accurate memory efficient single-pass method and a more general single-pass algorithm which, unlike previous methods, does not require prior information to assure near peak performance. Thirdly, combining ideas from subspace and single-pass algorithms, we present a more pass-efficient randomized block Krylov algorithm, which can achieve a desired accuracy using considerably fewer views than that needed by a subspace or previously studied block Krylov methods. However, the proposed accuracy enhanced block Krylov method is restricted to large matrices that are either accessed a few columns or rows at a time. Recommendations are also given on how to apply the subspace and block Krylov algorithms when estimating either the dominant left or right singular subspace of a matrix, or when estimating a normal matrix, such as those appearing in inverse problems. Computational experiments are carried out that demonstrate the applicability and effectiveness of the presented algorithms.
\end{abstract}



\section{Introduction}

Low-rank matrix approximation methods are important tools for improving computational performance when dealing with large data sets and highly parameterized computational models. Recent years have seen a surge in randomized methods for low-rank matrix approximation which are designed for high-performance computing on modern computer architectures and have provable accuracy bounds \cite{halko2011structure,mahoney2011,martinsson2016rsvd,liberty2007,martinsson2006,martinsson2011,rokhlin2009,tropp2017,woodruff2014,woolfe2008}. The main computational tasks of modern randomized matrix approximation algorithms involve multiplying the data matrix under consideration with a small number of tall, thin matrices. Since accessing a large matrix is expensive, limiting the number of times it is multiplied with another matrix is key to computational efficiency. With this in mind, randomized algorithms have been devised that only need to view the data matrix once \cite{halko2011structure,martinsson2006,tropp2017,woodruff2014,woolfe2008}. This is unlike classical matrix factorization algorithms which typically involve a serial process where the data matrix is multiplied numerous times with a vector. Randomized methods are therefore more suitable for high-performance computing than classical algorithms.

Randomized algorithms have quickly gained popularity because of their ability to accelerate many expensive linear algebra tasks. Another important reason is that many randomized algorithms are easy to implement, often requiring only a few lines of code in their most basic form. This study presents improved randomized algorithms for estimating a low-rank factorization of a matrix using an arbitrary number of views or passes over the information contained in the matrix. The algorithms presented here follow the trend of being simple and easy to implement. The focus of the study is on factorizing matrices based on truncated singular value decompositions. However, the ideas presented can also be extended to other matrix factorization methods. Though, the present study focuses on matrices with real value entries, the algorithms can be extended to the complex case like other related algorithms \cite{halko2011structure,tropp2017}.

First, we consider generalizing a popular randomized subspace iteration algorithm which uses an even number of matrix views. The more practical subspace iteration algorithm presented here works for any number of views $v \geq 2$. Next, to accommodate applications that can only afford a single matrix view, we expand upon and improve current state-of-the-art randomized single-pass methods. Finally, using ideas from our generalized subspace iteration method and single-pass methods, we present modified and pass-efficient randomized block Krylov methods. We discuss how the block Krylov approach can be made even more efficient when dealing with large matrices that, because of their size, are accessed a few rows at a time. These are all practical methods aimed at generating approximate, but high-quality, low-rank factorizations of large matrices. We also discuss subtle differences between applying randomized algorithms to a given matrix ${\bm J}$ or its transpose ${\bm J}^*$. This can be of interest when trying to estimate dominant singular subspaces or approximating so-called normal matrices, such as those appearing in inverse problems.

\subsection{Motivation}

The primary motivation for this study was to develop algorithms to speed up inverse methods used to estimate parameters in models describing subsurface flow in geothermal reservoirs \cite{osullivan2001state,bjarkason2017randomized}. Inverting models describing complex geophysical processes, such as fluid flow in the subsurface, frequently involves matching a large data set using highly parameterized computational models. Running the model commonly involves solving an expensive and nonlinear forward problem. Despite the possible nonlinearity of the forward problem, the link between the model parameters and simulated observations is often described in terms of a Jacobian matrix ${\bm J} \in \mathbb{R}^{N_d \times N_m}$, which locally linearizes the relationship between the parameters and observations. The size of ${\bm J}$ is therefore determined by the (large) parameter and observation spaces. In this case, explicitly forming ${\bm J}$ is out of the question since at best it involves solving $N_m$ direct problems (linearized forward simulations) or $N_d$ adjoint problems (linearized backward simulations) \cite{carrera2005,hinze2009,oliver2008,rodrigues2006adjoint}. Nevertheless, the information contained in ${\bm J}$ can be helpful for the purpose of inverting the model using nonlinear inversion methods such as a Gauss-Newton or Levenberg-Marquardt approach, and for quantifying uncertainty. 

Using adjoint simulation, direct simulation and randomized algorithms, the necessary information can be extracted from ${\bm J}$ without ever explicitly forming the large matrix ${\bm J}$. Bjarkason et al. \cite{bjarkason2017randomized} showed that inversion of a nonlinear geothermal reservoir model can be accelerated by using randomized low-rank methods coupled with adjoint and direct methods. Bjarkason et al. \cite{bjarkason2017randomized} presented a modified Levenberg-Marquardt approach which, at each inversion iteration, updates model parameters based on an approximate truncated singular value decomposition (TSVD) of ${\bm J}$. The TSVD of ${\bm J}$ was approximated using either a 1-view or 2-view randomized method. This involves evaluating ${\bm J}$ times a thin matrix and ${\bm J}^*$ times a thin matrix at every iteration. ${\bm J}$ times a thin matrix ${\bm H}$ is evaluated efficiently by solving a direct problem (linearized forward solve) for each column in ${\bm H}$. However, the advantage of the method is that each of these direct problems can be solved simultaneously. Similarly, ${\bm J}^*$ times a thin matrix is evaluated efficiently using an adjoint method (solving linearized backward problems). The advantage of using a 1-view method over a 2-view one is that all the direct and adjoint problems can be solved simultaneously in parallel. However, the 1-view method gives less accurate results when the randomized sampling is comparable in size to that used by the 2-view approach.

In \cite{bjarkason2017randomized}, randomized 1-view and 2-view methods were chosen since those methods use the fewest matrix accesses possible. However, in some cases it may be necessary to use more accurate low-rank approximation methods. This may be the case if the singular spectrum of ${\bm J}$ does not decay rapidly enough for the 1-view or 2-view methods to be suitably accurate. Then a more accurate randomized power or subspace iteration method \cite{erichson2017,gu2015,halko2011structure,rokhlin2009} can be used. Another option is to consider pass-efficient randomized block Krylov methods \cite{drineas2017,halko2011pca,martinsson2010normalized,musco2015krylov,rokhlin2009}. Standard randomized subspace iteration or block Krylov methods use $2(q + 1)$ views to form a low-rank approximation, where $q$ is the number of iterations. However, this leaves out the option of choosing an odd number of views. A more desirable algorithm would allow the user to specify a budget of $v$ views, which could be either odd or even. This is especially important for nonlinear inverse or uncertainty quantification problems, as the matrix views dominate the computational cost and $4$ or $6$ views could be considered too costly.

\subsection{Contributions and outline of the paper}

Inspired by the similarities between a subspace iteration method presented in the 1990's by \cite{vogel1994}, which uses an odd number of views $v \geq 3$, and modern randomized subspace iteration methods, we have developed a more general subspace iteration algorithm, see \cref{alg:gensubit}, which works for any number of views $v \geq 2$. \cref{alg:gensubit} returns an estimated rank-$p$ TSVD of the input matrix ${\bm A}$, given a target rank $p$ and number of views $v \geq 2$. The basic idea behind \cref{alg:gensubit} is that each additional application of the matrix ${\bm A}$ improves the randomized approximation. The expected gains in accuracy for each additional view are described by \cref{thm:powit,thm:genpowit}. These theorems show that the absolute gain in accuracy achieved by using an additional view is expected to decline exponentially with the number of views. Therefore, stopping at an odd number of matrix views may be sufficient and desirable to lower cost. For an even number of views, \cref{alg:gensubit} is equivalent to using a standard subspace iteration method. For an odd number of views it is similar to a modified randomized power iteration scheme proposed by \cite[Sect. 4.6]{rokhlin2009}, the difference being that \cref{alg:gensubit} applies a subspace iteration approach and this gives flexibility in terms of the number of matrix views.

Section \ref{sec:standardsubspace} briefly reviews the state-of-the-art for standard subspace iteration methods. Section \ref{sec:gensubspace} presents the more general subspace iteration method and gives theoretical bounds for the accuracy of the method. Section \ref{sec:apply2normalmat} discusses the difference between applying \cref{alg:gensubit} to a matrix ${\bm J}$ or ${\bm J}^*$, both from a perspective of computational cost and accuracy. By considering the low-rank approximation of a normal matrix ${\bm J}^* {\bm J}$, section \ref{sec:link2nystrompinched} also gives insight into the relationship between \cref{alg:gensubit} and algorithms designed for low-rank approximation of positive-semidefinite matrices. Applying \cref{alg:gensubit} one way can be shown to be algebraically equivalent to applying a Nystr\"{o}m type method \cite{gittens2016revisiting,halko2011structure} to ${\bm J}^* {\bm J}$, see section \ref{sec:link2nystrompinched}. Another way of using \cref{alg:gensubit} is algebraically equivalent to applying a pinched sketching method \cite{gittens2016revisiting,halko2011structure} to ${\bm J}^* {\bm J}$, see also section \ref{sec:link2nystrompinched}.

Section \ref{sec:lstsqinversion} outlines how low-rank approximations can be used to solve the type of nonlinear inverse problem motivating this study. Three variants for approximating Levenberg-Marquardt updates are discussed as well as the consequences of choosing certain combinations of model update strategies and low-rank approximation schemes.

To address applications that may only admit a budget of one view, section \ref{sec:1view} presents modified and improved randomized 1-view algorithms. The 1-view algorithms are mainly based on state-of-the-art algorithms proposed and discussed by \cite{tropp2017} but also draw upon algorithms suggested by \cite{woolfe2008,halko2011structure,martinsson2016rsvd,boutsidis2016,upadhyay2016,yu2017single}. Randomized 1-view algorithms are based on the idea that information gathered by simultaneously sketching parts of the column and row spaces of a matrix can suffice to form a reasonable low-rank approximation. 

The general 1-view algorithm proposed and recommended by \cite[Alg. 7]{tropp2017} has the drawback that the sketches for the column and row spaces cannot be chosen to have the same size without sacrificing robustness. Section \ref{sec:1ivewmodified} presents modified 1-view algorithms that enable the user to choose the same sketch size for the column and row sampling, without sacrificing robustness. We also discuss how these modified 1-view methods open up the possibility of post-processing the results from the matrix sketching to optimize the accuracy of the low-rank approximation. For the test matrices we considered, our modified algorithms and a simple post-processing scheme gives better overall performance than using the 1-view algorithm and sampling schemes recommended by \cite{tropp2017}. Additionally, based on the work of \cite{boutsidis2016,tropp2017,upadhyay2016}, section \ref{sec:1viewextended} provides a new and improved version of an extended 1-view algorithm, which for some applications may perform the best out of the presented 1-view methods in terms of memory use.

Section \ref{sec:blockKrylovMethods} outlines improved and more pass-efficient randomized block Krylov algorithms. The algorithms are especially aimed at problems where the multiplication with the data matrix and its transpose are expensive. This can be the case when dealing with large matrices stored out-of-core or Jacobian matrices appearing in large-scale inverse problems. Like the generalized subspace iteration method, we propose a Krylov algorithm that works for $v \geq 2$. We also suggest improvements for problems dealing with large matrices which are accessed a few rows at a time from out-of-core memory. In this case the Krylov approach can be up to twice as pass-efficient as previously proposed block Krylov methods \cite{drineas2010lstsq,halko2011structure,martinsson2010normalized,musco2015krylov,rokhlin2009}, by simultaneously sketching the column and row spaces of the input matrix whenever the matrix is viewed.

Section \ref{sec:results} gives experimental results to demonstrate the accuracy of the presented randomized algorithms and to support some of the claims made in previous sections. The randomized algorithms presented in this study have been coded using Python. The experimental Python code is available at https://github.com/ebjarkason/RandomSVDanyViews and includes code for the experiments considered in section \ref{sec:results}. Some of the Python algorithms were designed with adjoint and direct methods in mind. That is, they allow the user to specify functions for evaluating the action of the matrix of interest and its transpose on other matrices.

\section{Preliminaries}

Most of the notation used in the following sections follows standard practices. However, some of the notation needs clarifying. Methods are considered for finding a low-rank matrix that accurately approximates ${\bm A}$ in some sense. We use the spectral norm, denoted by $\norm{ \cdot }$, and the Frobenius norm, denoted by $\norm{ \cdot }_\text{F}$, to quantify the accuracy of the low-rank approximations.

This study focuses on approximation methods by way of singular value decomposition (SVD) of thin, low-rank matrices. The SVD factorization of a matrix ${\bm A} \in \mathbb{R}^{n_r \times n_c}$ can be written as
\begin{equation}
	{\bm A} = {\bm U} {\bm \Lambda} {\bm V}^* = \sum_{i=1}^N \lambda_i {\bm u}_i {\bm v}_i^*   ,
\end{equation}
where $N = \min(n_r , n_c)$. The matrices ${\bm U} = [ {\bm u}_1 \,  {\bm u}_2 \, \dots \, {\bm u}_N ]$ and ${\bm V} = [ {\bm v}_1 \,  {\bm v}_2 \, \dots \, {\bm v}_N ]$ have orthonormal columns, where ${\bm u}_i$ and ${\bm v}_i$ are the left and right singular vectors belonging to the $i$th singular value $\lambda_i$. The singular values $\lambda_1 \geq \lambda_2  \geq \dots \geq \lambda_N \geq 0$ are contained in the diagonal matrix ${\bm \Lambda} = \text{diag} [ \lambda_1 \, ,  \lambda_2 \, , \, \dots \, , \lambda_N ]$. We let ${\bm X}^*$ denote the conjugate transpose of the matrix ${\bm X}$. Almost all of the following discussion pertains to real valued matrices, and in that case the conjugate transpose is the same as the transpose.

For a matrix $\hat{\bm A}$ that is a rank-$p$ approximate TSVD of ${\bm A}$ we use
\begin{equation}
	\hat{\bm A} =\sum_{i=1}^p \lambda_i {\bm u}_i {\bm v}_i^*  = {\bm U}_p {\bm \Lambda}_p {\bm V}_p^*   .
\end{equation} 
Here $\lambda_i$, ${\bm u}_i$ and ${\bm v}_i$ can denote either the approximate or exact $i$th singular value and vectors of ${\bm A}$. It should be clear from the context which is being used. When using the exact values then $\hat{\bm A}$ is the optimal rank-$p$ approximation (in terms of the spectral and Frobenius norms), which is denoted by $\llbracket {\bm A} \rrbracket_p$. $[{{\bm U}}_p,\, {\bm \Lambda}_p,\, {\bm V}_p] = \tsvd({\bm A},\, p)$ is used to denote a function returning the SVD of $\llbracket {\bm A} \rrbracket_p$, i.e. the exact rank-$p$ TSVD of ${\bm A}$. Similarly, when $n_r = n_c$ and ${\bm A}$ is Hermitian positive-semidefinite then $[{\bm \Lambda}_p,\, {\bm V}_p] = \tevd({\bm A},\, p)$ is used to denote finding the eigenvalue decomposition of $\llbracket {\bm A} \rrbracket_p = {\bm V}_p {\bm \Lambda}_p {\bm V}_p^*$.

In what follows, the QR decomposition of ${\bm A} \in \mathbb{R}^{n_r \times n_c}$, with $n_r \geq n_c$, is defined as ${\bm Q} {\bm R} := {\bm A}$, where ${\bm Q}$ is an $n_r \times n_c$ matrix with orthonormal columns and ${\bm R}$ is an $n_c \times n_c$ upper triangular matrix. Finding such a thin or economic QR-factorization is expressed by $[{\bm Q},\, {\bm R}] = \qr({\bm A})$. Similarly, an $n_r \times r$ matrix ${\bm Q}$ which has orthonormal columns which form a basis for the range of a matrix ${\bm A}$ of rank $r$ is denoted by ${\bm Q} = \orth ({\bm A})$. In the presented algorithms we use $\randn(m, \, n)$ to denote an $m \times n$ Gaussian random matrix. For a Hermitian positive-definite matrix ${\bm A}$ we use ${\bm C}_L = \chol({\bm A},\, \mathit{LOWER})$ to express finding a lower-triangular Cholesky matrix ${\bm C}_L$ such that ${\bm A} = {\bm C}_L {\bm C}_L^*$.

\section{Standard randomized subspace iteration} \label{sec:standardsubspace}

A low-rank approximation of an $n_r \times n_c$ matrix ${\bm A}$ can be found by applying a simple two-stage randomized procedure \cite{erichson2017,halko2011structure,martinsson2016rsvd}. The first stage is a randomized range finder and the second stage uses the approximate range of ${\bm A}$ to construct the desired low-rank approximation.
\begin{enumerate}
\item (\textit{Randomized stage}) Find an $n_r \times k$ matrix ${\bm Q}_c$ whose columns are an approximate orthonormal basis for the range (column space) of ${\bm A}$, such that ${\bm A} \approx {\bm Q}_c {\bm Q}_c^* {\bm A}$.
\item (\textit{Deterministic stage}) Given the orthonormal matrix ${\bm Q}_c$, constructed during the first stage, evaluate ${\bm B} = {\bm Q}_c^* {\bm A}$. Then ${\bm Q}_c {\bm B}$, is a rank-$k$ approximation of ${\bm A}$. This approximate QB-decomposition can be used to construct other approximate low-rank factorizations of ${\bm A}$, such as the TSVD. 
\end{enumerate}

\begin{algorithm}[b!]
\caption{Randomized SVD using standard subspace iteration.} \label{alg:subit}
\begin{flushleft}
\textbf{INPUT:} Matrix ${\bm A} \in {\mathbb{R}}^{n_r \times n_c}$, integers $p > 0$, $l \geq 0$ and $q \geq 0$.\\
\textbf{RETURNS:} Approximate rank-$p$ SVD, ${\bm U}_p {\bm \Lambda}_p {\bm V}_p^*$, of ${\bm A}$.
\end{flushleft}
\begin{algorithmic}[1]
\STATE{${\bm \Omega}_r = \randn(n_c,\, p+l)$.}
\STATE{$[{\bm Q}_c,\, \sim] = \qr({\bm A} {\bm \Omega}_r)$.}
\FOR{$j = 1$ \textbf{to} $q$}
	\STATE{$[{\bm Q}_r,\, \sim] = \qr({\bm A}^* {\bm Q}_c)$.}
    \STATE{$[{\bm Q}_c,\, \sim] = \qr({\bm A} {\bm Q}_r)$.}
\ENDFOR
\STATE{$[{\bm Q}_r,\, {\bm R}_r] = \qr({\bm A}^* {\bm Q}_c)$.}
\STATE{$[\hat{{\bm V}}_p,\, {\bm \Lambda}_p,\, \hat{\bm U}_p] = \tsvd({\bm R}_r,\, p)$.}
\STATE{${\bm U}_p = {\bm Q}_c \hat{\bm U}_p$ and ${\bm V}_p = {\bm Q}_r \hat{\bm V}_p$.}
\end{algorithmic}
\end{algorithm}

The matrix ${\bm Q}_c$ can be generated cheaply by accessing or viewing the matrix ${\bm A}$ only once. Simply form a random matrix ${\bm \Omega}_r \in {\mathbb{R}}^{n_c \times k}$ and evaluate ${\bm Y}_c = {\bm A} {\bm \Omega}_r$, which sketches the range of ${\bm A}$. For the present study each element of the sampling matrix ${\bm \Omega}_r$ is drawn independently at random from a standard Gaussian distribution. However, other types of randomized sampling matrices can also be considered \cite{erichson2017,halko2011structure,szlam2014,tropp2017}. For robust outcomes $k$ should be chosen larger than the desired target rank $p$. That is, choose $k = p + l$, where $l$ is an oversampling factor required for increased accuracy (e.g., $l=10$ \cite{gu2015,halko2011structure,martinsson2016rsvd}). 

Subsequently, QR-decomposition can, for instance, be used to find the orthonormal matrix ${\bm Q}_c$, i.e. ${\bm Y}_c = {\bm Q}_c {\bm R}_c$. With this basis in hand, we can form the thin matrix ${\bm B}$ in Stage 2 by one additional view of ${\bm A}$. Taking the SVD of ${\bm B} = \hat{\bm U}_k {\bm \Lambda}_k {\bm V}_k^*$ and evaluating ${\bm U}_k = {\bm Q}_c \hat{\bm U}_k$, then ${\bm U}_k {\bm \Lambda}_k {\bm V}_k^*$ is an approximate rank-$k$ TSVD of ${\bm A}$. For the desired rank-$p$ TSVD, the $l$ smallest singular values of ${\bm B}$, which are estimated least accurately, can be truncated away to give ${\bm A} \approx {\bm Q}_c \llbracket {\bm B} \rrbracket_p = {\bm U}_p {\bm \Lambda}_p {\bm V}_p^*$. This gives the basic randomized SVD algorithm \cite{gu2015,martinsson2016rsvd}, which is a 2-view method.

The basic randomized 2-view method works well when the singular spectrum decays rapidly. However, it may lack accuracy for some applications. In that case power or subspace iteration can be used to improve the randomized approach \cite{erichson2017,gu2015,halko2011structure,rokhlin2009}. The key idea is that classic power iteration can be used to find an improved orthonormal basis ${\bm Q}_c$ during Stage 1 of the random procedure. Using $q$ power iterations an orthonormal matrix ${\bm Q}_c$ is found using $({\bm A} {\bm A}^*)^q {\bm A} {\bm \Omega}_r$ instead of ${\bm A} {\bm \Omega}_r$. A drawback is that for each iteration the matrix ${\bm A}$ is accessed two additional times. The spectral norm error of the randomized approximation ${\bm A} \approx {\bm Q}_c {\bm Q}_c^* {\bm A}$ is expected to improve exponentially with $q$ \cite{halko2011structure,rokhlin2009}, see the following \cref{thm:powit} from \cite{halko2011structure}.

\begin{theorem}[Average spectral error for the power scheme \cite{halko2011structure}]\label{thm:powit}
Let ${\bm A} \in {\mathbb{R}}^{n_r \times n_c}$ with nonnegative singular values $\lambda_1 \geq \lambda_2 \geq \cdots$, let $p \geq 2$ be the target rank and let $l \geq 2$ be an oversampling parameter, with $p + l \leq \min (n_r, n_c)$. Draw a Gaussian random matrix ${\bm \Omega}_r \in {\mathbb{R}}^{n_c \times (p+l)}$ and set ${\bm Y}_c = ({\bm A} {\bm A}^*)^q {\bm A} {\bm \Omega}_r$ for an integer $q \geq 0$. Let ${\bm Q}_c \in {\mathbb{R}}^{n_r \times (p+l)}$ be an orthonormal matrix which forms a basis for the range of $\,{\bm Y}_c$. Then
\begin{displaymath}
    \mathbb{E} \left[ \norm{ {\bm A} - {\bm Q}_c {\bm Q}_c^* {\bm A} } \right] 
    \leq
    \left[ 
    \left( 1 + \sqrt{\frac{p}{l-1}} \right) \lambda_{p+1}^{2q+1} +
    \frac{\mathrm{e} \sqrt{p+l}}{l} 
    \left( \sum_{j>p} \lambda_j^{2(2q+1)} \right)^{1/2}
    \right]^{1/(2q+1)}
    \, .
\end{displaymath}
\end{theorem}

In floating point arithmetic, subspace iteration is a numerically more robust extension of the power iteration scheme. Subspace iteration is typically recommended for this reason, while it is equivalent to power iteration in exact arithmetic.
\cref{alg:subit} gives a pseudocode for a standard randomized subspace iteration method for estimating a TSVD of a matrix. \cref{alg:subit} is based on algorithms proposed by \cite{voronin2016}. During the subspace iteration, renormalization is performed using QR-decomposition after each application of the input matrix ${\bm A}$ or its transpose. In \cref{alg:subit} the TSVD of ${\bm B} = {\bm Q}_c^* {\bm A}$ is found by QR-decomposition of its transpose ${\bm B}^* = {\bm Q}_r {\bm R}_r$. Taking the TSVD of the small $(p + l) \times (p + l) $ matrix ${\bm R}_r$ then $\llbracket {\bm B} \rrbracket_p = \llbracket {\bm R}_r^* \rrbracket_p {\bm Q}_r^*$.

Computational efficiency can be improved by using LU-decomposition instead of QR-decomposition to renormalize after each of the first $q$ applications of ${\bm A}$ and the first $q$ applications of ${\bm A}^*$ \cite{erichson2017,szlam2014}. Another option is to apply a subspace/power iteration hybrid and skip some of the QR-factorizations \cite{szlam2014,voronin2016}. These options may sacrifice some accuracy for improved efficiency. The algorithms presented here use QR decomposition after each application of ${\bm A}$ or ${\bm A}^*$ during the subspace iteration, because the main topic of the present study concerns the number of times the matrix ${\bm A}$ is viewed.

The standard subspace iteration method, discussed in this section, always views the input matrix ${\bm A}$ an even $2(q+1)$ times. To suit any given budget of matrix views, a randomized low-rank approximation algorithm that works for any positive integer number of matrix views is appealing. The following section presents a more general randomized subspace iteration algorithm, which gives an approximate TSVD of a matrix for any number of views greater than one. Section \ref{sec:1view} discusses algorithms for low-rank approximation of a matrix when the budget is only one matrix view.

\section{Generalized randomized subspace iteration} \label{sec:gensubspace}

The basic idea of the generalized subspace iteration method is that the standard subspace iteration process can be halted halfway through an iteration to give an orthonormal matrix ${\bm Q}_r$ which approximates the co-range (row space) of the matrix of interest ${\bm A}$. In terms of the power iteration process this equates to evaluating the co-range sketch ${\bm Y}_r = ({\bm A}^* {\bm A})^q {\bm \Omega}_r$ for $q \geq 1$, which is like performing $q-1/2$ power iterations. Then, with an orthonormal basis ${\bm Q}_r$ for the range of ${\bm Y}_r$, the matrix ${\bm A}$ can be approximated as ${\bm A} \approx {\bm A} {\bm Q}_r {\bm Q}_r^*$. This approach gives a rank-$p$ approximation $\llbracket {\bm A} \rrbracket_p \approx \llbracket {\bm A} {\bm Q}_r \rrbracket_p {\bm Q}_r^*$ using $2q + 1$ views.

\begin{algorithm}[b!]
\caption{Randomized SVD using generalized subspace iteration.} \label{alg:gensubit}
\begin{flushleft}
\textbf{INPUT:} Matrix ${\bm A} \in {\mathbb{R}}^{n_r \times n_c}$, integers $p > 0$, $l \geq 0$ and $v \geq 2$.\\
\textbf{RETURNS:} Approximate rank-$p$ SVD, ${\bm U}_p {\bm \Lambda}_p {\bm V}_p^*$, of ${\bm A}$.
\end{flushleft}
\begin{algorithmic}[1]
\STATE{${\bm Q}_r = \randn(n_c,\, p+l)$.} \label{line:subiterRandn}
\FOR{$j = 1$ \textbf{to} $v$}
	\IF{$j$ is odd}
    	\STATE{$[{\bm Q}_c,\, {\bm R}_c] = \qr({\bm A} {\bm Q}_r)$.} \label{line:contentiousOddViews}
    \ELSE
        \STATE{$[{\bm Q}_r,\, {\bm R}_r] = \qr({\bm A}^* {\bm Q}_c)$.}
     \ENDIF
\ENDFOR
\IF{$v$ is even}
	\STATE{$[\hat{{\bm V}}_p,\, {\bm \Lambda}_p,\, \hat{\bm U}_p] = \tsvd({\bm R}_r,\, p)$.}
\ELSE
	\STATE{$[\hat{{\bm U}}_p,\, {\bm \Lambda}_p,\, \hat{\bm V}_p] = \tsvd({\bm R}_c,\, p)$.}
\ENDIF
\STATE{${\bm U}_p = {\bm Q}_c \hat{\bm U}_p$ and ${\bm V}_p = {\bm Q}_r \hat{\bm V}_p$.}
\end{algorithmic}
\end{algorithm}

The generalized subspace iteration algorithm is given by \cref{alg:gensubit}. Given a budget of $v \geq 2$ views, \cref{alg:gensubit} returns an approximate TSVD ${\bm U}_p {\bm \Lambda}_p {\bm V}_p^*$ of ${\bm A}$ by one of the following approaches:
\begin{enumerate}
\item (\textit{If $v$ is even}) Find an orthonormal matrix ${\bm Q}_c$ whose columns form a basis for the range of $({\bm A} {\bm A}^*)^{(v - 2)/2} {\bm A} {\bm \Omega}_r$. Then find the rank-$p$ TSVD ${\bm V}_p {\bm \Lambda}_p \hat{\bm U}_p^*$ of ${\bm A}^* {\bm Q}_c$ and set ${\bm U}_p = {\bm Q}_c \hat{\bm U}_p$.
\item (\textit{If $v$ is odd}) Find an orthonormal matrix ${\bm Q}_r$ whose columns form a basis for the range of $({\bm A}^* {\bm A})^{(v - 1)/2} {\bm \Omega}_r$. Then find the rank-$p$ TSVD ${\bm U}_p {\bm \Lambda}_p \hat{\bm V}_p^*$ of ${\bm A} {\bm Q}_r$ and set ${\bm V}_p = {\bm Q}_r \hat{\bm V}_p$.
\end{enumerate}
For an even number of matrix views \cref{alg:gensubit} simply proceeds by applying the standard subspace iteration method outlined in the previous section. This is easy to verify by comparing \cref{alg:subit} and \cref{alg:gensubit} for $v = 2(q + 1)$. \cref{alg:gensubit} is a slight modification of \cref{alg:subit} to allow an odd number of matrix views. For an odd number of matrix views \cref{alg:gensubit} is similar to and uses the same number of matrix views as a modified power iteration method proposed by \cite[Sect. 4.6]{rokhlin2009}. The main difference is that the algorithm presented here uses the subspace iteration framework to give a practical algorithm that works for any budget of matrix views $v \geq 2$.

The generalized subspace iteration algorithm presented here was motivated by a subspace iteration algorithm proposed by \cite{vogel1994} for estimating a TSVD of a rectangular matrix using an odd number of matrix views. The subspace iteration method presented by \cite{vogel1994} has all the ingredients needed for a modern randomized subspace iteration algorithm and only minor modifications are needed to arrive at the subspace iteration methods presented here. Similarities between the subspace iteration method of \cite{vogel1994} and a modern randomized 2-view method are discussed in \cite{bjarkason2017randomized}.

\subsection{Accuracy of \cref{alg:gensubit}}

For an even number of views, \cref{alg:gensubit} equates to using the standard subspace iteration method. Then the expected approximation error is given by \cref{thm:powit}. For an odd number of matrix views (greater than one) the expected approximation error of \cref{alg:gensubit} is given by the following \cref{thm:genpowit}. A proof for \cref{thm:genpowit} is given in \cref{sec:proof}. The proof is based on the approach used in \cite{halko2011structure} to prove \cref{thm:powit}.

\begin{theorem}[Average spectral error for half-power scheme]\label{thm:genpowit}
Let ${\bm A} \in {\mathbb{R}}^{n_r \times n_c}$ with nonnegative singular values $\lambda_1 \geq \lambda_2 \geq \cdots$, let $p \geq 2$ be the target rank and let $l \geq 2$ be an oversampling parameter, with $p + l \leq \min (n_r, n_c)$. Draw a Gaussian random matrix ${\bm \Omega}_r \in {\mathbb{R}}^{n_c \times (p+l)}$ and set ${\bm Y}_r = ({\bm A}^* {\bm A})^q {\bm \Omega}_r$ for an integer $q \geq 1$. Let ${\bm Q}_r \in {\mathbb{R}}^{n_c \times (p+l)}$ be an orthonormal matrix which forms a basis for the range of $\,{\bm Y}_r$. Then
\begin{displaymath}
    \mathbb{E} \left[ \norm{ {\bm A} - {\bm A} {\bm Q}_r {\bm Q}_r^* } \right]
    \leq
    \left[ 
    \left( 1 + \sqrt{\frac{p}{l-1}} \right) \lambda_{p+1}^{2q} +
    \frac{\mathrm{e} \sqrt{p+l}}{l} 
    \left( \sum_{j>p} \lambda_j^{4q} \right)^{1/2}
    \right]^{1/(2q)}
    \, .
\end{displaymath}
\end{theorem}
The bounds in \cref{thm:genpowit} are what we could expect by extrapolating from \cref{thm:powit}. That is, the expected error of the randomized approximation decays exponentially with the number of matrix views. \cref{thm:powit} states that the expected approximation error for ${\bm A} \approx {\bm Q}_c {\bm Q}_c^* {\bm A}$ depends on $2q+1$ when using $2q+1$ views to generate an approximate basis ${\bm Q}_c$ for the range of ${\bm A}$. \cref{thm:genpowit} states similarly that the expected error for ${\bm A} \approx {\bm A} {\bm Q}_r {\bm Q}_r^*$ depends on $2q$ when using $2q$ views to form an approximate basis ${\bm Q}_r$ for the co-range of ${\bm A}$. \cref{thm:powit,thm:genpowit} show that it is important for the accuracy of \cref{alg:gensubit} that the spectrum of ${\bm A}$ decays rapidly.

\subsection{Apply \cref{alg:gensubit} to a matrix or its transpose?} \label{sec:apply2normalmat}

Given a budget of views $v$, an approximate TSVD of some matrix ${\bm J}$ can be found using \cref{alg:gensubit} with either ${\bm J}$ or ${\bm J}^*$ as input. However, it may be worth considering that these two approaches can differ in computational cost. Another thing to consider is that, depending on the application, the choice of method can have different impacts on the accuracy of downstream applications, which apply information from the low-rank approximation. The latter may seem surprising considering that the expected quality for an approximate TSVD of ${\bm J}$ is the same whether we apply \cref{alg:gensubit} to ${\bm J}$ or ${\bm J}^*$.

\subsubsection{Computational cost} \label{sec:gensubcost}

Considering computational cost, the choice between applying \cref{alg:gensubit} to ${\bm J}$ or ${\bm J}^*$ depends on a few factors. For an even number of views the two options only differ when it comes to generating the random Gaussian sampling matrix in Line \ref{line:subiterRandn} of \cref{alg:gensubit}. The cost of generating the sampling matrix is $(p + l) n_c T_\text{rand}$, where $T_\text{rand}$ is the cost of generating a Gaussian random number. In that case, to reduce computational cost ${\bm J}^*$ should be used as input when ${\bm J}$ has more columns than rows, otherwise use ${\bm J}$.

However, for an odd number of views the computational cost that should be considered when choosing between applying \cref{alg:gensubit} to ${\bm J}$ or ${\bm J}^*$ is $\mathcal{O}([p + l] n_c T_\text{rand} + (p + l) T_\text{mult} + [p + l]^2 n_r)$, where $T_\text{mult}$ indicates the cost of multiplying the input matrix by a vector and the last term comes from the QR-factorization on Line \ref{line:contentiousOddViews} in \cref{alg:gensubit}. For the type of problem motivating this study the dominant cost is from the matrix multiplication and in some cases there may be a noticeable difference between the cost of evaluating ${\bm J}$ or ${\bm J}^*$ times a matrix. For physics-based simulations ${\bm J}$ can represent a Jacobian matrix, which gives a local linear mapping between the model inputs (parameters) and outputs (e.g., observations or predictions). Then ${\bm J}$ times a matrix can be evaluated efficiently using a direct method (linearized forward simulations) and ${\bm J}^*$ times a matrix can be evaluated using an adjoint method (linearized backward simulations). The adjoint runs require more memory and can therefore be more costly for large simulation applications. In that case it might be best to apply \cref{alg:gensubit} to ${\bm J}$, when $v$ is odd, to lower the cost. This is likewise the case when using automatic differentiation \cite{bucker2006automatic,griewank2008autodiffbook} where the forward and backward modes correspond to the direct and adjoint method, respectively. The memory costs associated with tracking backwards (using a reverse mode or adjoint method) are commonly reduced somewhat by applying a method called checkpointing \cite{griewank1992}.

\subsubsection{Computational accuracy} \label{sec:JorJTaccuracy}

To evaluate accuracy we can look at the bounds in \cref{thm:powit,thm:genpowit}. For an even number of matrix views $v$ \cref{alg:gensubit} gives matrices ${\bm Q}_c$ and ${\bm Q}_r$ which form approximate bases for the range and co-range of the input matrix ${\bm A}$, such that
\begin{align}
    \mathbb{E} \left[ \norm{ {\bm A} - {\bm Q}_c {\bm Q}_c^* {\bm A} } \right] 
&    \leq
    \left[ 
    \left( 1 + \sqrt{\frac{p}{l-1}} \right) \lambda_{p+1}^{v-1} +
    \frac{\mathrm{e} \sqrt{p+l}}{l} 
    \left( \sum_{j>p} \lambda_j^{2(v-1)} \right)^{1/2}
    \right]^{1/(v-1)} ;
\\
    \mathbb{E} \left[ \norm{ {\bm A} - {\bm A} {\bm Q}_r {\bm Q}_r^* } \right]
&   \leq
    \left[ 
    \left( 1 + \sqrt{\frac{p}{l-1}} \right) \lambda_{p+1}^{v} +
    \frac{\mathrm{e} \sqrt{p+l}}{l} 
    \left( \sum_{j>p} \lambda_j^{2v} \right)^{1/2}
    \right]^{1/v} .  
\end{align}
However, for an odd number of matrix views $v$ the error bounds for the final matrices ${\bm Q}_c$ and ${\bm Q}_r$ in \cref{alg:gensubit} are
\begin{align}
    \mathbb{E} \left[ \norm{ {\bm A} - {\bm Q}_c {\bm Q}_c^* {\bm A} } \right] 
&    \leq
    \left[ 
    \left( 1 + \sqrt{\frac{p}{l-1}} \right) \lambda_{p+1}^{v} +
    \frac{\mathrm{e} \sqrt{p+l}}{l} 
    \left( \sum_{j>p} \lambda_j^{2v} \right)^{1/2}
    \right]^{1/v} ;
\\
    \mathbb{E} \left[ \norm{ {\bm A} - {\bm A} {\bm Q}_r {\bm Q}_r^* } \right]
&   \leq
    \left[ 
    \left( 1 + \sqrt{\frac{p}{l-1}} \right) \lambda_{p+1}^{v-1} +
    \frac{\mathrm{e} \sqrt{p+l}}{l} 
    \left( \sum_{j>p} \lambda_j^{2(v-1)} \right)^{1/2}
    \right]^{1/(v-1)}   .
\end{align}
Noting that the left and right singular vectors are estimated as a linear combination of the columns of ${\bm Q}_r$ and ${\bm Q}_c$ according to ${\bm U}_p = {\bm Q}_c \hat{\bm U}_p$ and ${\bm V}_p = {\bm Q}_r \hat{\bm V}_p$, the above error bounds suggest that the accuracy of the left-singular vectors may differ to that of the right-singular vectors. When $v$ is even ${\bm V}_p$ should be more accurate than or as accurate as ${\bm U}_p$, but when $v$ is odd ${\bm U}_p$ should be more accurate than or as accurate as ${\bm V}_p$.

This distinction may be important where the goal is not only the low-rank approximation of a matrix ${\bm J}$, but for instance if we are mainly interested in its right-singular vectors and do not care (or care less) about the left-singular vectors. Section \ref{sec:lstsqinversion} discusses an area of application where this can be of interest. One such application is estimating a truncated eigen-decomposition (TEVD) of the normal matrix ${\bm J}^* {\bm J}$. In that case we need the leading right-singular vectors and singular values of ${\bm J}$. With an even number of views $v$, then it can be more accurate to apply \cref{alg:gensubit} to ${\bm J}$. However, for an odd $v$ it is better to input ${\bm J}^*$. The following subsection shows that this strategy equates to applying a popular low-rank approximation method, meant for positive-semidefinite (psd) matrices, to the normal matrix ${\bm J}^* {\bm J}$. Recent studies by \cite{drineas2017,saibaba2018} consider the accuracy of approximate principal subspaces generated by the standard randomized subspace \cite{saibaba2018} and block Krylov \cite{drineas2017} methods.

\begin{algorithm}[b!]
\caption{Construct approximate range or co-range.} \label{alg:qrorqc}
\begin{flushleft}
\textbf{INPUT:} Matrix ${\bm A} \in {\mathbb{R}}^{n_r \times n_c}$, integers $p > 0$, $l \geq 0$ and $v \geq 1$.\\
\textbf{RETURNS:} Approximate basis for the column, ${\bm Q}_c$ ($v$ is odd), or row space, ${\bm Q}_r$ ($v$ is even), of ${\bm A}$. \\
\textbf{function} $\QrQc ( {\bm A},\, p,\, l,\, v )$
\end{flushleft}
\begin{algorithmic}[1]
\STATE{${\bm Q}_r = \randn(n_c,\, p+l)$.}
\FOR{$j = 1$ \textbf{to} $v$}
	\IF{$j$ is odd}
    	\STATE{$[{\bm Q}_c,\, \sim] = \qr({\bm A} {\bm Q}_r)$.}
    \ELSE
        \STATE{$[{\bm Q}_r,\, \sim] = \qr({\bm A}^* {\bm Q}_c)$.}
     \ENDIF
\ENDFOR
\IF{$v$ is even}
	\STATE{{\bf return} ${\bm Q}_r$}
\ELSE
	\STATE{{\bf return} ${\bm Q}_c$}
\ENDIF
\end{algorithmic}
\end{algorithm}

\begin{algorithm}[b!]
\caption{Nystr\"{o}m approach for a truncated eigen-decomposition of ${\bm J}^* {\bm J}$.} \label{alg:nystrom}
\begin{flushleft}
\textbf{INPUT:} Matrix ${\bm J} \in {\mathbb{R}}^{n_r \times n_c}$, integers $p > 0$, $l \geq 0$ and $v \geq 2$.\\
\textbf{RETURNS:} Approximate rank-$p$ EVD, ${\bm V}_p {\bm \Lambda}^2_p {\bm V}_p^*$, of ${\bm J}^* {\bm J}$.
\end{flushleft}
\begin{algorithmic}[1]
\IF{$v==2$}
	\STATE{${\bm Q}_r = \orth( \randn(n_c,\, p+l) )$.}
\ELSIF{$v$ is even}
	\STATE{${\bm Q}_r = \QrQc ({\bm J},\, p,\, l,\, v-2)$.}
\ELSE
	\STATE{${\bm Q}_r = \QrQc ({\bm J}^*,\, p,\, l,\, v-2)$.}
\ENDIF
\STATE{${\bm Y}_r = {\bm J}^* [{\bm J} {\bm Q}_r]$.}
\STATE{$\nu = 2.2 \cdot 10^{-16} \norm{{\bm Y}_r}$.}
\STATE{${\bm Y}_r \leftarrow {\bm Y}_r + \nu {\bm Q}_r$.}
\STATE{${\bm B} = {\bm Q}_r^* {\bm Y}_r$.}
\STATE{${\bm C}_L = \chol( [{\bm B} + {\bm B}^*]/2,\, \mathit{LOWER})$.}
\STATE{${\bm F} = {\bm C}_L^{-1} {\bm Y}_r^*$.}
\STATE{$[\sim ,\, \hat{\bm \Lambda}_p,\, {\bm V}_p,] = \tsvd({\bm F},\, p)$.}
\STATE{${\bm \Lambda}_p^2 = \hat{\bm \Lambda}_p^2 - \nu {\bm I}$ and set all negative elements of ${\bm \Lambda}_p^2$ to $0$.}
\end{algorithmic}
\end{algorithm}

\subsubsection{Link to standard methods for approximating a normal matrix} \label{sec:link2nystrompinched}

The so-called prolonged or Nystr\"{o}m type, and pinched sketching methods have been discussed extensively in the literature for approximating a psd matrix ${\bm A}_\text{psd}$, see, e.g., \cite{gittens2016revisiting,halko2011structure}. Given an approximate range ${\bm Q}_r$ for ${\bm A}_\text{psd}$, the pinched method gives the following approximation
\begin{align} \label{eq:basicPinched}
{\bm A}_\text{psd} \approx {\bm Q}_r ( {\bm Q}_r^* {\bm A}_\text{psd} {\bm Q}_r ) {\bm Q}_r^*
.
\end{align}
The prolonged method, on the other hand, uses
\begin{equation} \label{eq:basicProlonged}
{\bm A}_\text{psd}
\approx 
( {\bm A}_\text{psd}  {\bm Q}_r )
[ {\bm Q}_r^* {\bm A}_\text{psd}  {\bm Q}_r ]^{-1}
( {\bm A}_\text{psd}  {\bm Q}_r )^*
.
\end{equation}
Assuming ${\bm Q}_r$ and a straightforward implementation then the pinched and prolonged approaches have the same cost when it comes to multiplications with ${\bm A}_\text{psd}$, i.e. both use ${\bm A}_\text{psd}  {\bm Q}_r$. Though, in this case they have about the same computational cost, the prolonged method is more accurate than the pinched one \cite{gittens2016revisiting,halko2011structure}.

A psd matrix can generally be written as a normal matrix ${\bm A}_\text{psd} = {\bm J}^* {\bm J}$. Then the pinched sketch can be written as
\begin{align} \label{eq:genpinched}
{\bm J}^* {\bm J} \approx {\bm Q}_r ( {\bm Q}_r^* {\bm J}^* {\bm J} {\bm Q}_r ) {\bm Q}_r^* 
= ({\bm Q}_r  {\bm Q}_r^* {\bm J}^*) ({\bm J} {\bm Q}_r  {\bm Q}_r^*)
.
\end{align}
Let us, assume for now that ${\bm Q}_r$ is formed in a standard fashion by applying $q$ power iterations to ${\bm A}_\text{psd}$, with ${\bm Q}_r = \orth ({\bm A}_\text{psd}^q {\bm \Omega}_r)$. Then \eqref{eq:genpinched} is algebraically (but not numerically) equivalent to applying \cref{alg:gensubit} to ${\bm J}$ using $2q + 1$ views to find ${\bm J} \approx {\bm J} {\bm Q}_r  {\bm Q}_r^*$ and then forming the low-rank TEVD of ${\bm J}^* {\bm J}$.

For the prolonged sketch we can write
\begin{align} \label{eq:gennystrom}
{\bm J}^* {\bm J}
&\approx 
( {\bm J}^* {\bm J}  {\bm Q}_r )
[ {\bm Q}_r^* {\bm J}^* {\bm J} {\bm Q}_r ]^{-1}
( {\bm J}^* {\bm J} {\bm Q}_r )^*
\\
&=
{\bm J}^* ( {\bm J}  {\bm Q}_r )
[ ( {\bm J} {\bm Q}_r )^* ( {\bm J}  {\bm Q}_r )]^{-1}
( {\bm J} {\bm Q}_r )^* {\bm J}
\\
&=
{\bm J}^* {\bm Q}_c {\bm Q}_c^*  {\bm J}
=
({\bm J}^* {\bm Q}_c {\bm Q}_c^*) ({\bm Q}_c   {\bm Q}_c^*  {\bm J})  ,  \label{eq:gennystrom2}
\end{align}
where ${\bm Q}_c$ is an orthonormal basis for ${\bm J} {\bm Q}_r$. Therefore, the standard prolonged method is like using the generalized subspace iteration method for the approximation ${\bm J}  \approx {\bm Q}_c {\bm Q}_c^*  {\bm J}$. Assuming again that ${\bm Q}_r$ is generated by applying $q$ power iterations to ${\bm A}_\text{psd}$, then the standard prolonged scheme is algebraically equivalent to applying \cref{alg:gensubit} to ${\bm J}$ using $2q + 2$ views. This is a generalization of the observation made by \cite{gittens2016revisiting}, that for ${\bm Q}_r = \orth ({\bm A}_\text{psd} {\bm \Omega}_r)$ and $q=1$ the prolonged method is like applying a four view method to ${\bm A}_\text{psd}^{1/2}$ and the pinched method is like viewing ${\bm A}_\text{psd}^{1/2}$ three times. Therefore, it is clear why the standard prolonged (or Nystr\"{o}m) type method is considered more accurate than the pinched method as it is like applying one extra view or half iteration with the generalized subspace iteration method. Furthermore, the standard prolonged method corresponds to using \cref{alg:gensubit} to estimate ${\bm J}^* {\bm J}$ in the more accurate way using $2q+2$ views (see section \ref{sec:JorJTaccuracy}). The pinched method, however, corresponds to using \cref{alg:gensubit} to estimate ${\bm J}^* {\bm J}$ in the less accurate way with $2q+1$ views.

For the above discussion on the standard pinched and prolonged methods the factorization ${\bm A}_\text{psd} = {\bm J}^* {\bm J}$ was mainly used to demonstrate differences between \eqref{eq:basicPinched} and \eqref{eq:basicProlonged}, and it was not assumed that ${\bm J}$ could be applied separately. However, if we can multiply ${\bm J}^*$ and ${\bm J}$ separately with matrices, then the pinched and prolonged methods can be implemented more generally given a budget of views $v$ for ${\bm J}$. 

\begin{algorithm}[b!]
\caption{Using pinched sketch for a truncated eigen-decomposition of ${\bm J}^* {\bm J}$.} \label{alg:pinched}
\begin{flushleft}
\textbf{INPUT:} Matrix ${\bm J} \in {\mathbb{R}}^{n_r \times n_c}$, integers $p > 0$, $l \geq 0$ and $v \geq 2$.\\
\textbf{RETURNS:} Approximate rank-$p$ EVD, ${\bm V}_p {\bm \Lambda}^2_p {\bm V}_p^*$, of ${\bm J}^* {\bm J}$.
\end{flushleft}
\begin{algorithmic}[1]
\IF{$v$ is even}
	\STATE{${\bm Q}_r = \QrQc ({\bm J}^*,\, p,\, l,\, v-1)$.}
\ELSE
	\STATE{${\bm Q}_r = \QrQc ({\bm J},\, p,\, l,\, v-1)$.}
\ENDIF
\STATE{${\bm B}_c = {\bm J} {\bm Q}_r$.}
\STATE{$[{\bm \Lambda}_p^2,\, \hat{\bm V}_p] = \tevd({\bm B}_c^* {\bm B}_c,\, p)$.}
\STATE{${\bm V}_p = {\bm Q}_r \hat{\bm V}_p$.}
\end{algorithmic}
\end{algorithm}

The modified prolonged or Nystr\"{o}m approach is given in \cref{alg:nystrom}, which is based on a single-pass Nystr\"{o}m algorithm proposed by \cite{tropp2017psd} for approximating psd matrices. \cref{alg:nystrom} broadens their algorithm to work for our purposes. \cref{alg:nystrom} uses $v-2$ views of ${\bm J}$ to generate an approximate basis ${\bm Q}_r$ for the co-range of ${\bm J}$, by calling \cref{alg:qrorqc}. Then, a TEVD is formed based on \eqref{eq:gennystrom}, which costs an additional 2 views. Considering these steps for the modified Nystr\"{o}m approach and \eqref{eq:gennystrom2}, we see that this is like using the recommended accurate way of using \cref{alg:gensubit} for approximating a normal matrix, see section \ref{sec:JorJTaccuracy}. 

Similarly, a modified pinched algorithm which uses $v$ views for approximating a normal matrix is given by \cref{alg:pinched}. \cref{alg:pinched} uses $v-1$ views to form ${\bm Q}_r$ instead of $v-2$ views. Using the final view an approximation is formed based on \eqref{eq:genpinched}. Notice that this approach equates to not following the recommendations given in section \ref{sec:JorJTaccuracy} for using \cref{alg:gensubit} to approximate ${\bm J}^* {\bm J}$. Section \ref{sec:resultsnormalmatrix} demonstrates that the modified prolonged method is more accurate than the pinched approach. As mentioned, we can use these methods interchangeably with the subspace iteration \cref{alg:gensubit} to estimate ${\bm J}^* {\bm J}$, when ${\bm J}$ and ${\bm J}^*$ can be multiplied separately with matrices. We have preferred using \cref{alg:gensubit} since it allows the flexibility of generating an approximate TSVD of ${\bm J}$ or ${\bm J}^* {\bm J}$ as needed.

\subsection{Modifications of the generalized subspace iteration algorithm}

\subsubsection{Normalizing the half iterations}

The generalized subspace iteration method (\cref{alg:gensubit}) can be varied in many of the same ways as the standard subspace iteration method (\cref{alg:subit}). Like \cref{alg:subit}, \cref{alg:gensubit} can be modified to use LU-factorizations instead of QR-factorizations to normalize the algorithm after each of the first $v-2$ views. Another option may be to skip some of the renormalization steps. These types of alternative renormalization schemes are presented for the standard subspace or power iteration approaches in \cite{erichson2017,szlam2014,voronin2016}.

\subsubsection{Different post-processing}

\cref{alg:subit} was based on an algorithm proposed by \cite{voronin2016} and the same algorithm inspired in part \cref{alg:gensubit}. Voronin and Martinsson \cite{voronin2016} proposed another subspace iteration algorithm that avoids finding the TSVD of the long matrix ${\bm B}$ from the QB-factorization approximating ${\bm A}$. That approach involves finding a TEVD of the smaller normal matrix ${\bm B} {\bm B}^*$ and a few more steps give the desired approximate TSVD. The generalized subspace iteration method can be modified to use this type of post-processing, see \cref{alg:gensubitV2}. \cref{alg:gensubitV2} can be compared with \cref{alg:pinched} to better understand the parallels between the pinched sketching scheme and the generalized subspace iteration approach. Note that \cref{alg:gensubitV2} requires extra care when zero valued singular values appear.

\begin{algorithm}[b!]
\caption{Randomized SVD using generalized subspace iteration V2} \label{alg:gensubitV2}
\begin{flushleft}
\textbf{INPUT:} Matrix ${\bm A} \in {\mathbb{R}}^{n_r \times n_c}$, integers $p > 0$, $l \geq 0$ and $v \geq 2$.\\
\textbf{RETURNS:} Approximate rank-$p$ SVD, ${\bm U}_p {\bm \Lambda}_p {\bm V}_p^*$, of ${\bm A}$.
\end{flushleft}
\begin{algorithmic}[1]
\IF{$v$ is even}
    \STATE{${\bm Q}_c = \QrQc ({\bm A},\, p,\, l,\, v-1)$.}
    \STATE{${\bm B}_r = {\bm A}^* {\bm Q}_c$.}
	\STATE{$[{\bm \Lambda}_p^2 ,\, \hat{{\bm U}}_p] = \tevd({\bm B}_r^* {\bm B}_r,\, p)$.}
    \STATE{${\bm U}_p = {\bm Q}_c \hat{\bm U}_p$ and ${\bm V}_p = {\bm B}_r \hat{\bm U}_p {\bm \Lambda}_p^{-1}$.}
\ELSE
	\STATE{${\bm Q}_r = \QrQc ({\bm A},\, p,\, l,\, v-1)$.}
    \STATE{${\bm B}_c = {\bm A} {\bm Q}_r$.}
	\STATE{$[{\bm \Lambda}_p^2 ,\, \hat{{\bm V}}_p] = \tevd({\bm B}_c^* {\bm B}_c,\, p)$.}
    \STATE{${\bm V}_p = {\bm Q}_r \hat{\bm V}_p$ and ${\bm U}_p = {\bm B}_c \hat{\bm V}_p {\bm \Lambda}_p^{-1}$.}
\ENDIF
\end{algorithmic}
\end{algorithm}

\subsubsection{Check for convergence}

\cref{alg:gensubit} was also inspired by a subspace iteration method proposed in the 1990's by \cite{vogel1994}. For deciding when to halt their subspace iteration method, \cite{vogel1994} considered the change in estimated singular values between iterations. Their algorithm is halted when the singular values have not perceivably changed (given some tolerance) relative to the ones estimated at the previous iteration. A similar stopping condition could be considered for \cref{alg:gensubit}. All it would need is to evaluate the singular values of ${\bm R}_r$ or ${\bm R}_c$ at every (half) iteration (not costly since they are small matrices).

\subsubsection{Other random sampling matrices}

Other randomized sampling matrices can be considered instead of the Gaussians used here for initializing the generalized subspace iterations. Alternative types of random matrices have been discussed in detail in the literature \cite{erichson2017,gittens2016revisiting,halko2011structure,szlam2014,tropp2017}. The type of sampling matrix impacts the computational cost. Then a cost examination like the one in section \ref{sec:gensubcost} can be considered, but with modifications. For example, for a dense matrix ${\bm A}$ a subsampled randomized Fourier transform (SRFT) matrix can be used instead of a Gaussian to reduce the cost of evaluating ${\bm A} {\bm \Omega}_r$ from $\mathcal{O}(n_r n_c [p+l])$ to $\mathcal{O}(n_r n_c \log{[p+l]})$ \cite{halko2011structure, liberty2007, woolfe2008}. However, choosing an SRFT typically requires that ${\bm \Omega}_r$ have more columns than would be needed when using a Gaussian sampling matrix. Therefore, we have not considered using SRFTs since the cost of applying a Jacobian to a sampling matrix using adjoint and direct methods depends on the number of columns the sampling matrix has. We have not considered using sparse random sampling matrices for the same reason, though they cost less to store and generate.

\subsection{Other matrix factorization methods}

The present study focuses on methods for generating approximate TSVD factorizations. Nevertheless, the general ideas presented here can also be applied to other low-rank methods that use an approximate range or co-range to form a factorization. A partial pivoted QR-factorization can, for instance, be formed using a QB-factorization \cite{martinsson2016qrblocked}. Instead of using $q$ standard power or subspace iterations to form the approximate basis ${\bm Q}_c$ using $2q + 1$ views, ${\bm Q}_c$ can be formed using $v-1$ views. Then ${\bm B} = {\bm Q}_c^* {\bm A}$ is formed with the last view. Erichson et al. \cite{erichson2018nonnegative} used such a matrix ${\bm B}$, generated by standard subspace iteration, to initialize their proposed method for generating a nonnegative matrix factorization. For more flexibility, the half subspace iteration approach could also be used within their scheme. Martinsson \cite{martinsson2016rsvd} presented randomized CUR and interpolative decomposition (ID) algorithms that use power iterations. Those power iterations could also be replaced more generally by a half subspace iteration process. 

In \cite{martinsson2017utv}, a randomized algorithm is described for generating a UTV-factorization ${\bm A} = {\bm U} {\bm T} {\bm V}^*$ (where ${\bm T}$ is a triangular matrix, and ${\bm U}$ and ${\bm V}$ are unitary matrices). An inner loop of their algorithm, contains a power-iteration procedure to generate a co-range sketch ${\bm Y}_r$ of ${\bm A}$. A more general half power or subspace iteration approach could also be used in this case for generating ${\bm Y}_r$. Martinsson et al. \cite{martinsson2017utv} noted the parallels between their UTV procedure and a randomized power iteration method. In fact, the inner loop of their UTV algorithm involves $2q + 3$ views and the accuracy of their procedure can be considered in terms of half power iterations \cite[Thm. 1]{martinsson2017utv}. By applying a half power/subspace iteration method, their procedure could be generalized to use $v \geq 3$ views, though \cite{martinsson2017utv} did not propose this option.

\section{Application: nonlinear least-squares inversion} \label{sec:lstsqinversion}

Solving linear matrix equations is a routine task in scientific computing which becomes computationally more demanding as the matrices involved increase in dimension. Various randomized methods have been presented for solving linear matrix equations or linear regression \cite{avron2010blendenpik,clarkson2013,drineas2010lstsq,gower2015randomized,meng2014,rokhlin2008regression,sarlos2006,xiang2013}. The algorithms and methods presented in this study were motivated by a study \cite{bjarkason2017randomized} which considered randomized methods for updating model parameters of a nonlinear reservoir model using a modified Levenberg-Marquardt (LM) approach. Another topic of interest is using randomized methods to estimate the posterior probability distribution for the parameters of a nonlinear model using a quadratic or Laplace approximation \cite{buithanh2012,cui2014,cui2016}. Assuming Gaussian statistics for the observation noise and prior, this involves approximating a normal matrix ${\bm J}^* {\bm J}$.

\subsection{Levenberg-Marquardt update}

An assumption commonly used for inversion of reservoir models is that the distributions for the observations and the prior of the model parameters can be described by Gaussian statistics. That is, the statistics for the observations and prior can be described by expected mean values and a covariance matrix. Without loss of generality, we apply a linear whitening transformation such that both the model parameters ${\bm x} \in \mathbb{R}^{N_m}$ and observation error ${\bm d} ({\bm x}) \in \mathbb{R}^{N_d}$, are associated with a zero mean and a covariance matrix that is an appropriately sized identity matrix. ${\bm d} ({\bm x})$ quantifies how much the simulated outputs, given ${\bm x}$, depart from the observations. Then the inversion task is to solve
\begin{equation} \label{eq:lstsqproblem}
\operatorname*{argmin}_{\bm x} 
\, 
\norm{ {\bm d} ({\bm x}) }^2
+
\mu \norm{ {\bm x} }^2
,
\end{equation}
where $\mu$ is a regularization weight.

Using the LM approach, $\eqref{eq:lstsqproblem}$ is tackled by an iterative procedure where model updates are found by
\begin{equation} \label{eq:LMupdate}
	\left[ {\bm J}^* {\bm J}  + (\mu + \gamma ) {\bm I}\right]  \delta {\bm x}  =  - {\bm J}^* {\bm d}  - \mu {\bm x}   
    ,
\end{equation}
and new models are generated by ${\bm x} \leftarrow {\bm x} + \delta {\bm x}$. Here $\gamma > 0$ is the adjustable LM damping factor and ${\bm J} \in \mathbb{R}^{N_d \times N_m}$ is the Jacobian matrix which defines an approximate local linear mapping from ${\bm x}$ to ${\bm d} ({\bm x})$. In \cite{bjarkason2017randomized,shirangi2014, shirangi2016,tavakoli2010} ${\bm J}$ is referred to as the dimensionless sensitivity matrix. For a large number of observations $N_d$ and $N_m$, generating the often dense Jacobian and solving \eqref{eq:LMupdate} becomes costly. However, for a nonlinear problem an exact solution to \eqref{eq:LMupdate} is not necessarily needed or desired, and approximate solution procedures are typically used to save time. After applying a randomized TSVD method to find ${\bm J} \approx {\bm U}_p {\bm \Lambda}_p {\bm V}_p^*$, then that low-rank factorization can be used to approximately solve \eqref{eq:LMupdate} \cite{bjarkason2017randomized}. It makes sense to apply randomized low-rank approximation methods to ${\bm J}$ since for many problems it has a rapidly decaying spectrum.

\subsection{Approximate truncated updates}

Here we discuss three approximate ways of solving the LM update based on a TSVD of the Jacobian matrix. Some of the analysis is based on the work of \cite{voronin2015lstsq} who looked at using randomized low-rank approximation methods for solving a seismic tomography problem, which is a linear inverse problem. Some of the theoretical analysis in \cite{voronin2015lstsq} assumes that the TSVD is exact. Here we try to provide analysis that is also applicable to an approximate TSVD, but indicate when the conclusions use the exact TSVD.

The work of \cite{bjarkason2017randomized} discussed advantages of using randomized 2-view or 1-view methods for speeding up the LM approach for inverting geothermal reservoir models. The approaches presented in \cite{bjarkason2017randomized} are based on the work of \cite{shirangi2014,shirangi2016,tavakoli2010}, who used an iterative Lanczos approach to generate a TSVD of ${\bm J}$. After forming an approximate TSVD of ${\bm J}$, an approximate LM update can be found according to \cite{bjarkason2017randomized,shirangi2014,shirangi2016,tavakoli2010}
\begin{equation} \label{eq:LMupdatedx1}
\delta {\bm x}_1 = - {\bm V}_p [{\bm \Lambda}_p^2 + (\mu + \gamma) {\bm I}_p]^{-1}  ({\bm \Lambda}_p {\bm U}_p^* {\bm d} + \mu {\bm V}_p^* {\bm x})
= \sum_{i=1}^p \alpha_i {\bm v}_i
,
\end{equation}
where 
\begin{equation}
\alpha_i = - \frac{1}{\mu + \gamma + \lambda_i^2}  \left[  {\lambda}_i {\bm u}_i^*   {\bm d} +  \mu {\bm v}_i^* {\bm x} \right]
.
\end{equation}
Following \cite{shirangi2014,shirangi2016,tavakoli2010}, \cite{bjarkason2017randomized} chose to gradually increase $p$ between LM iterations, since using a large $p$ at early iterations can be computationally wasteful and can result in bad convergence characteristics. The study by \cite{bjarkason2017randomized} demonstrated that using randomized low-rank approximations can work well and can be considerably faster than using a standard iterative Lanczos method. 

The LM method is better suited to the randomized paradigm than using the corresponding Gauss-Newton method ($\gamma = 0$), since the LM damping factor acts as a fail-safe mechanism. That is, when a LM update fails because of an inaccurate low-rank approximation, then a new update is attempted using a larger $\gamma$. Increasing $\gamma$ reduces contributions from the least accurately estimated 
singular vectors, which are associate with small singular values. A benefit of using \eqref{eq:LMupdatedx1}
is that no extra views of ${\bm J}$ are needed for additional damping factors $\gamma$ within a LM iteration, since the TSVD of ${\bm J}$ is independent of $\gamma$.

A second option similar to \eqref{eq:LMupdatedx1} is to use \cite{tavakoli2011}
\begin{equation} \label{eq:LMupdatedx2}
\delta {\bm x}_2 = - {\bm V}_p [{\bm \Lambda}_p^2 + (\mu + \gamma) {\bm I}_p]^{-1} {\bm V}_p^* ({\bm g}_\text{obs} + \mu {\bm x}) 
= \sum_{i=1}^p \beta_i {\bm v}_i
,
\end{equation}
where
\begin{equation}
\beta_i = 
- \frac{1}{\mu + \gamma + \lambda_i^2}  \left[  {\bm v}_i^*   {\bm g}_\text{obs} +  \mu {\bm v}_i^* {\bm x} \right]
\end{equation}
and ${\bm g}_\text{obs} = {\bm J}^* {\bm d}$ is the observation gradient. Generating the additional observation gradient does not require much additional computational cost since, with minor modifications, it can be evaluated along with the first pass for ${\bm J}^*$ within the randomized scheme. Equation \eqref{eq:LMupdatedx2} needs estimates for the principal right-singular vectors and singular values of ${\bm J}$. These can be evaluated by approximating a TSVD of ${\bm J}$ or a TEVD of ${\bm J}^* {\bm J}$, see discussion on this in section \ref{sec:apply2normalmat}.

When using the exact TSVD then $\delta {\bm x}_1 = \delta {\bm x}_2$, since
\begin{align}
\delta {\bm x}_2 &= 
- {\bm V}_p [{\bm \Lambda}_p^2 + (\mu + \gamma) {\bm I}_p]^{-1} {\bm V}_p^* ({\bm V} {\bm \Lambda} {\bm U}^* {\bm d} + \mu {\bm x}) 
\\
&= 
- {\bm V}_p [{\bm \Lambda}_p^2 + (\mu + \gamma) {\bm I}_p]^{-1}  ({\bm \Lambda}_p {\bm U}_p^* {\bm d} + \mu {\bm V}_p^* {\bm x}) 
= \delta {\bm x}_1
.
\end{align}
Likewise, when using a randomized approximation according to ${\bm J} \approx \hat{\bm J} = \llbracket {\bm J} {\bm Q}_r \rrbracket_p {\bm Q}_r^* = {\bm U}_p {\bm \Lambda}_p \hat{\bm V}_p^* {\bm Q}_r^* = {\bm U}_p {\bm \Lambda}_p {\bm V}_p^*$, with ${\bm V}_p = {\bm Q}_r \hat{\bm V}_p$, then
\begin{align}
\delta {\bm x}_2 &= 
- {\bm V}_p [{\bm \Lambda}_p^2 + (\mu + \gamma) {\bm I}_p]^{-1} {\bm V}_p^* ( {\bm J}^* {\bm d} + \mu {\bm x}) 
\label{eq:dx1dx2same}
\\
&= 
- {\bm V}_p [{\bm \Lambda}_p^2 + (\mu + \gamma) {\bm I}_p]^{-1} (\hat{\bm V}_p^* {\bm Q}_r^* {\bm J}^* {\bm d} + \mu {\bm V}_p^* {\bm x}) 
\\
&= 
- {\bm V}_p [{\bm \Lambda}_p^2 + (\mu + \gamma) {\bm I}_p]^{-1} (\hat{\bm V}_p^* [{\bm J} {\bm Q}_r]^* {\bm d} + \mu {\bm V}_p^* {\bm x}) 
\\
&= 
- {\bm V}_p [{\bm \Lambda}_p^2 + (\mu + \gamma) {\bm I}_p]^{-1} (\hat{\bm V}_p^* \hat{\bm V}_{p} {\bm \Lambda}_{p} {\bm U}_{p}^* {\bm d} + \mu {\bm V}_p^* {\bm x}) 
\\
&= 
- {\bm V}_p [{\bm \Lambda}_p^2 + (\mu + \gamma) {\bm I}_p]^{-1} ({\bm \Lambda}_p {\bm U}_p^* {\bm d} + \mu {\bm V}_p^* {\bm x}) 
= \delta {\bm x}_1
.
\end{align}
In this case it is not worth generating the observation gradient explicitly and using $\delta {\bm x}_1$ should be preferred. However, when using ${\bm J} \approx \hat{\bm J}  = {\bm Q}_c \llbracket {\bm Q}_c^* {\bm J} \rrbracket_p = {\bm Q}_c \hat{\bm U}_p {\bm \Lambda}_p {\bm V}_p^*$ then $\delta {\bm x}_1$ uses ${\bm g}_\text{obs} \approx \hat{\bm J}^* {\bm d} = \llbracket {\bm J}^* {\bm Q}_c \rrbracket_p  {\bm Q}_c^* {\bm d}$ and $\delta {\bm x}_1$ can differ from $\delta {\bm x}_2$. Finding $\delta {\bm x}_2$ with ${\bm J} \approx {\bm Q}_c \llbracket {\bm Q}_c^* {\bm J} \rrbracket_p$ uses the recommended way (when considering accuracy) of applying \cref{alg:gensubit} or \cref{alg:nystrom} when approximating ${\bm J}^* {\bm J}$, see section \ref{sec:apply2normalmat}. Using $\delta {\bm x}_2$ with \cref{alg:gensubit} to generate $\bm J \approx \llbracket {\bm J} {\bm Q}_r \rrbracket_p {\bm Q}_r^*$ is like using $\delta {\bm x}_2$ with \cref{alg:pinched} to approximate ${\bm J}^* {\bm J}$.

A third approximate solution to \eqref{eq:LMupdate} can be found by using ${\bm J} \approx \hat{\bm J} = {\bm U}_p {\bm \Lambda}_p {\bm V}_p^*$ and
\begin{equation} \label{eq:approxPosteriorLaplace}
\left[ {\bm J}^* {\bm J}  + (\mu + \gamma ) {\bm I}\right]^{-1}
\approx 
\left[ \hat{\bm J}^* \hat{\bm J} + (\mu + \gamma ) {\bm I}\right]^{-1}
=
\frac{1}{\mu + \gamma} \left[ {\bm I} - {\bm V}_p {\bm D}_p {\bm V}_p^* \right]
,
\end{equation}
where ${\bm D}_p \in \mathbb{R}^{p \times p}$ is a diagonal matrix with $[{\bm D}_p]_{i,i} = \lambda_i^2 /(\mu + \gamma + \lambda_i^2)$. 
Taking $\gamma = 0$, then \eqref{eq:approxPosteriorLaplace} can be used to estimate the posterior parameter covariance matrix and can therefore be used to estimate parameter uncertainty \cite{buithanh2012,petra2014}. Using \eqref{eq:approxPosteriorLaplace} for the LM update gives
\begin{align} \label{eq:LMupdatedx3}
\delta {\bm x}_3 = - \frac{1}{\mu + \gamma} \left[ {\bm I} - {\bm V}_p {\bm D}_p {\bm V}_p^* \right] ({\bm g}_\text{obs} + \mu {\bm x})  
= - \frac{1}{\mu + \gamma} \left[ ({\bm g}_\text{obs} + \mu {\bm x}) + \sum_{i=1}^p \lambda_i^2 \beta_i {\bm v}_i \right]
.
\end{align}
\Cref{eq:LMupdatedx3} was proposed by \cite{vogel1993} with ${\bm J}^* {\bm J}$ approximated using a a subspace iteration procedure applied to ${\bm J}^* {\bm J}$. Similarly, \cite{wang2016spsd} proposed using this type of approximation and randomized methods to solve linear equations arising in Gaussian process regression and classification in the machine learning context.

The approximation error of \eqref{eq:approxPosteriorLaplace} is 
$
\mathcal{O} \left( \sum_{p+1}^{N_m} \frac{\lambda_i^2}{\mu + \gamma + \lambda_i^2}  \right)
$
\cite{buithanh2012,petra2014}. Therefore, a good approximation can be maintained when only truncating away singular values that are negligible compared to $\sqrt{\mu + \gamma}$. This matches our experience with using \eqref{eq:LMupdatedx3}, where updates typically fail unless a LM damping factor is used such that $\mu + \gamma$ is about or greater than an order of magnitude above that of $\lambda_p$.

\cref{sec:appendixTSVDLMupdates} compares the properties of the three approximate LM solution methods presented here, based on an analysis given by \cite{voronin2015lstsq}. A comparison is given for the length of the model updates and some insight is given for the expected approximation error. The end of \cref{sec:appendixTSVDLMupdates} also gives a correction for a proposition given by \cite[Prop. 5.6]{voronin2015lstsq}. As shown in \cref{sec:appendixTSVDLMupdates}, for a given low-rank approximation the lengths of $\delta {\bm x}_1$ and $\delta {\bm x}_2$ increase monotonically with the truncation $p$. Therefore, when using $\delta {\bm x}_1$ or $\delta {\bm x}_2$, the truncation point $p$ can be used to regularize each model update. However, the third scheme $\delta {\bm x}_3$ does not have this property since lowering $p$ is expected to result in a longer model update $\delta {\bm x}_3$. Experimenting with synthetic inverse problems we have found that this is the case in practice. Because of the regularizing properties of $\delta {\bm x}_1$ and $\delta {\bm x}_2$ we have found that those methods work well at early iterations. However, $\delta {\bm x}_3$ can work better at later iterations since the model update $\delta {\bm x}_3$ is in the full parameter space, unlike $\delta {\bm x}_1$ and $\delta {\bm x}_2$ which only base the model update on the leading $p$ right-singular vectors. We have found that choosing $\delta {\bm x}_3$ at late inversion iterations can help to reduce both the overall objective function \eqref{eq:lstsqproblem} and the regularization term. We have attributed this to the fact that $\delta {\bm x}_3$ works in the full space, which can help smooth out deviations from the prior model that may build up during the nonlinear inversion, as well as improving matches to observations.

\section{Single-pass methods for rectangular matrices} \label{sec:1view}

The work of 
\cite{bjarkason2017randomized} also considered using a 1-view method to generate a TSVD of a dimensionless sensitivity matrix and then solving the LM update equations approximately using \eqref{eq:LMupdatedx1}. This approach speeds up each LM iteration, since all the adjoint and direct equations can in this case be solved in parallel. 1-view methods are broadly of interest for problems where a standard 2-view approach is considered too costly, or where the information in the data matrix is only accessible a single time and a 1-view method is the only option. The following subsections consider advancing state-of-the-art 1-view methods.

\subsection{Streaming models}
In recent years, low-rank approximation methods have been developed that only need to access the matrix of interest once \cite{halko2011structure,tropp2017,woolfe2008,boutsidis2016,upadhyay2016,cohen2015,clarkson2009,woodruff2014}. These single-pass or 1-view methods assemble range and co-range sketches with one pass of the data matrix ${\bm A}$ and then use the information contained in the sketches to generate a low-rank approximation. This section discusses 1-view methods for rectangular matrices. Note that when the matrix of interest is Hermitian or positive-semidefinite, then 1-view algorithms that are optimized for such matrices \cite{halko2011structure,tropp2017,tropp2017psd} should be considered. 

1-view methods are of interest in a streaming setting where the data matrix ${\bm A}$ is never stored fully in random-access memory but is presented as a finite stream of linear updates \cite{tropp2017}: 
\begin{equation} \label{eq:streammat}
{\bm A} = \sum_{i=1}^N {\bm H}_i
.
\end{equation}
Noting that we discard each ${\bm H}_i$ after its use; as each innovation matrix ${\bm H}_i$ is briefly made available, we can gradually sample from the column and row space of ${\bm A}$ by
\begin{equation} \label{eq:streammatmult}
{\bm A} {\bm \Omega}_r = \sum_{i=1}^N {\bm H}_i {\bm \Omega}_r  
\,\,\,
\text{and}
\,\,\,
{\bm A}^* {\bm \Omega}_c = \sum_{i=1}^N {\bm H}_i^* {\bm \Omega}_c ,
\end{equation}
where ${\bm \Omega}_r \in  {\mathbb{R}}^{n_c \times k_r}$ and ${\bm \Omega}_c \in  {\mathbb{R}}^{n_r \times k_c}$ are random sampling matrices. For example, we might want a low-rank approximation of a large matrix that needs to be stored out of core and because of time-constraints the elements of the out-of-core matrix can only be accessed once. In that case ${\bm H}_i$ could represent a sparse matrix containing a few elements of ${\bm A}$.

For the application motivating the present study, we seek a low-rank approximation of an $n_r$ by $n_c$ Jacobian matrix ${\bm A} = {\bm J}$, where ${\bm J} {\bm \Omega}_r$ and ${\bm J}^* {\bm \Omega}_c$ can be evaluated simultaneously in parallel by solving $k_c$ direct and $k_r$ adjoint problems \cite{bjarkason2017randomized}. Similar to the streaming model \cref{eq:streammat,eq:streammatmult}, for a transient simulation problem (omitting some details) the direct and adjoint methods essentially boil down to evaluating
\begin{equation} \label{eq:directadjointstream}
{\bm J} {\bm \Omega}_r = \sum_{i=1}^{N_t} {\bm C}_i \left[ {\bm J}_x^i \right]^{-1} {\bm B}_i 
\quad 
\text{and} 
\quad
{\bm J}^* {\bm \Omega}_c = \sum_{i=0}^{N_t - 1} [{\bm G}_{N_t -i}]^* \left( \left[ {\bm J}_x^{N_t -i} \right]^* \right)^{-1} \hat{\bm B}_{N_t -i} ,
\end{equation}
where $N_t$ is the number of simulation time-steps and ${\bm J}_x^i$ is the Jacobian matrix for the forward equations at the $i$th simulation time. Unlike the general \cref{eq:streammat,eq:streammatmult}, the direct and adjoint methods are restricted to tracking forward and backward in time, respectively.

\subsection{The baseline 1-view method} \label{sec:1viewbaselinealg}

Recently, \cite{tropp2017} presented a new 1-view algorithm for fixed-rank approximation of rectangular matrices and compared it with other state-of-the-art randomized 1-view methods. Tropp et al. \cite{tropp2017} compared their 1-view method \cite[Alg. 7]{tropp2017} with two similar 1-view sketching methods based on the work of \cite{cohen2015,woodruff2014} and a more complicated, extended 1-view method based on the work of \cite{boutsidis2016,upadhyay2016}, which uses an additional sketch generated by applying subsampled randomized Fourier transform (SRFT) sampling matrices. The approaches based on \cite{cohen2015} and \cite{woodruff2014} performed the worst. Tropp et al. \cite{tropp2017} concluded that their 1-view algorithm \cite[Alg. 7]{tropp2017} is the preferred method for matrices that have a rapidly decaying spectrum and can therefore be approximated accurately with a low-rank approximation. However, for a small amount of oversampling or matrices with flat singular spectra, then the extended method based on \cite{boutsidis2016,upadhyay2016} performed the best, for a given memory budget. However, the extended SRFT sketching method \cite{tropp2017} can be improved so that it performs as well as the 1-view method suggested by \cite{tropp2017} for matrices with rapidly decaying spectra, see sections \ref{sec:1viewextended} and \ref{sec:1viewresults}.

Here we consider the algorithm recommended by \cite{tropp2017} as the baseline 1-view method. The following subsections look at possible improvements of the baseline algorithm. The baseline 1-view algorithm \cite[Alg. 7]{tropp2017} is as follows:
\begin{enumerate}
\item Given oversampling parameters $l_1$ and $l_2$ with $0 \leq l_1 \leq l_2$, form random sampling matrices ${\bm \Omega}_r \in  {\mathbb{R}}^{n_c \times (p + l_1)}$ and ${\bm \Omega}_c \in  {\mathbb{R}}^{n_r \times (p + l_2)}$ for the range and co-range.
\item Find the range and co-range sketches ${\bm Y}_c = {\bm A} {\bm \Omega}_r$ and ${\bm Y}_r = {\bm A}^* {\bm \Omega}_c$.
\item Find an orthonormal basis for the range ${\bm Q}_c = \orth ({\bm Y}_c)$.
\item Find the least-squares solution ${\bm X}$ to $({\bm \Omega}_c^* {\bm Q}_c) {\bm X} = {\bm Y}_r^*$, i.e. ${\bm X} = ({\bm \Omega}_c^* {\bm Q}_c)^{\dagger} {\bm Y}_r^*$.
\item Then ${\bm A} \approx {\bm Q}_c \llbracket {\bm X} \rrbracket_p$.
\end{enumerate}
The above 1-view method is based on the same elements as the basic 2-view method. The main difference is that in the 1-view method the co-range sketch ${\bm Y}_r$ is constructed independent of the range sketch ${\bm Y}_c$. Therefore, the range and co-range sketches, which are used to form the low-rank approximation, can be found using one pass of the data matrix ${\bm A}$. For the type of Jacobian matrix motivating the present study, this allows us to form a low-rank approximation by solving all the necessary adjoint and direct problems simultaneously in parallel \cite{bjarkason2017randomized}. However, the convenience of constructing an approximation with only one view impacts accuracy. Like \cite{tropp2017}, we generally recommend more accurate methods, such as \cref{alg:gensubit}, for applications that can afford more than one view.

\subsection{Baseline oversampling schemes} \label{sec:baselineoversampling}

For practical applications we need schemes for determining oversampling parameters that work well in practice. For a sketch size budget $T = 2p + l_1 + l_2$, \cite{tropp2017} suggested three oversampling schemes based on empirical evidence and theoretical bounds for the accuracy of their 1-view approach. Tropp et al. \cite{tropp2017} presented oversampling schemes to suit both real valued matrices and complex matrices. However, for the following discussion we assume that we are dealing with a real matrix and $T \geq 2p + 6$.

First, for a budget $T$ and a relatively flat singular spectrum \cite{tropp2017} recommended selecting the oversampling parameters as
\begin{equation} \label{eq:overflat}
l_1 = 
\max \left\{
2 ,
\left\lfloor
(T - 1)
\frac{\sqrt{p (T - p - 2)(1 - 2/(T - 1))} - (p - 1)}{T - 2p - 1}
\right\rfloor
- p
\right\}
\,\,\, \text{and} \,\,\,
l_2 = T - 2p - l_1 .
\end{equation}
Secondly, for a moderately decaying spectrum \cite{tropp2017} recommended using
\begin{equation} \label{eq:overmed}
l_1 = 
\max \left\{
2 ,
\left\lfloor
(T - 1)/3
\right\rfloor
- p
\right\}
\,\,\, \text{and} \,\,\,
l_2 = T - 2p - l_1 .
\end{equation}
Thirdly, when dealing with a matrix that has a rapidly decaying spectrum \cite{tropp2017} recommended choosing 
\begin{equation} \label{eq:overrapid}
l_1 = 
\left\lfloor
(T - 2)/2
\right\rfloor
- p
\,\,\, \text{and} \,\,\,
l_2 = T - 2p - l_1 .
\end{equation}
When using the baseline 1-view method, it is generally not advisable to choose $l_1$ too close to $l_2$ ($l_1 \approx l_2$), e.g. by following \cref{eq:overrapid}. Using $l_1 \approx l_2$ and especially $l_1 =l_2$ is a bad idea, unless the spectral decay is rapid. For peak performance, some prior knowledge about the matrix at hand is needed to determine which of the above oversampling schemes should be chosen. This can be the case for certain applications. However, when reliable prior ideas are unavailable for choosing the oversampling scheme then \cref{eq:overmed} is the favored all-around choice \cite{tropp2017}.

\subsection{Disadvantages of the baseline 1-view method}

Needing prior insights into the data matrix to achieve peak performance and not being able to use $l_1 \approx l_2$ without the risk of losing accuracy are drawbacks which would be good to remedy. Furthermore, using the most versatile oversampling scheme \cref{eq:overmed} is suboptimal for matrices that have rapidly decaying spectra. To obtain the best performance of the 1-view method we should choose $l_1$ close to $l_2$ in cases where ${\bm A}$ has sharp spectral decay, but we need an escape mechanism for cases where the spectral decay is not sharp.

Another reason for wanting to fix the drawbacks of choosing $l_1 = l_2$ is that for some applications $l_2$ can be a computational bottleneck. In that case it is tempting to set $l_1 = l_2$ to try to get as much information as possible given the computational constraints. However, in terms of memory the 1-view method is limited by the need to store the random sampling matrices and the sketching matrices. Tropp et al. \cite{tropp2017} focused on the storage aspect and therefore compared methods in terms of the sketch size budget $T = 2p + l_1 + l_2$, since the storage cost is proportional to $T$. 

The following subsections discuss modifications to the baseline 1-view method to enable good performance for a range of problems with the choice $l_1 = l_2$. In particular, the modifications involve a parameter which can be chosen by post-processing the information contained in the randomized sketches to get close to peak algorithm performance for a range of matrices.

\subsection{A more flexible 1-view method} \label{sec:1ivewmodified}

Choosing $l_1 \approx l_2$ commonly results in solving a badly posed or badly conditioned least squares problem $({\bm \Omega}_c^* {\bm Q}_c) {\bm X} = {\bm Y}_r^*$ and as a result the low-rank approximation can suffer. The coefficient matrix ${\bm \Omega}_c^* {\bm Q}_c$ has $p + l_2$ rows and $p + l_1$ columns, and therefore a way to avoid a badly conditioned coefficient matrix is to sample more for the co-range ($l_2 > l_1$) to give an overdetermined and better conditioned problem. Another option is to follow the methods presented in \cite{woolfe2008,halko2011structure,martinsson2016rsvd} and select ${\bm Q}_c$ based on a TSVD of the range sketch. 

Instead of choosing ${\bm Q}_c$ as an orthonormal basis for the full range of ${\bm Y}_c$, we can choose $[{\bm Q}_c, \sim, \sim] = \tsvd({\bm Y}_c,\, p + l_\text{c})$, where $0 \leq l_\text{c} \leq l_1$. That is, ${\bm Q}_c \in {\mathbb{R}}^{n_r \times (p + l_\text{c})}$ contains the $p + l_\text{c}$ leading left singular vectors of ${\bm Y}_c$. If $l_\text{c}$ is smaller than $l_1$, then using $({\bm \Omega}_c^* {\bm Q}_c) \in {\mathbb{R}}^{(p + l_2) \times (p + l_\text{c})}$ can be expected to give a better conditioned problem, even for $l_1 = l_2$. Finding the left-singular vectors of the range sketch ${\bm Y}_c$ instead of just finding an orthonormal basis for the range sketch does not add much to the cost. Especially considering that applying ${\bm A}$ (to find range and co-range sketches) typically contributes the most to the computational cost and time for applications using the 1-view approach. 

\begin{algorithm}[b!]
  \caption{Randomized 1-view method for TSVD: Tropp variant.}\label{alg:1view}
  \begin{flushleft}
  \textbf{INPUT:} {Matrix ${\bm A} \in {\mathbb{R}}^{n_r \times n_c}$, integers $p > 0$ and $l_2 \geq l_1 \geq l_\text{c} \geq 0$}.\\
  \textbf{RETURNS:} {Approximate rank-$p$ SVD, ${\bm U}_p {\bm \Lambda}_p {\bm V}_p^*$, of ${\bm A}$.}
  \end{flushleft}
  \begin{algorithmic}[1]
  	\STATE{\textbf{a}) ${\bm \Omega}_r = \randn(n_c,\, p+l_1)$ and \textbf{b}) ${\bm \Omega}_c = \randn(n_r,\, p+l_2)$.}
    \STATE{Optional: \textbf{a}) ${\bm \Omega}_r =\text{orth}({\bm \Omega}_r)$ and \textbf{b}) ${\bm \Omega}_c =\text{orth}({\bm \Omega}_c)$.}
    \STATE{\textbf{a}) ${\bm Y}_c = {\bm A} {\bm \Omega}_r$ and \textbf{b}) ${\bm Y}_r = {\bm A}^* {\bm \Omega}_c$.}
    \IF{($l_\text{c} < l_1$)}
    	\STATE{$[{\bm Q}_c,\, {\bm R}_c] = \qr({\bm Y}_c)$.}
    	\STATE{$[\hat{\bm Q}_c,\, \sim,\, \sim] = \tsvd({\bm R}_c,\, p+l_\text{c})$.}
    	\STATE{${\bm Q}_c = {\bm Q}_c \hat{\bm Q}_c$.}
    \ELSE
    	\STATE{$[{\bm Q}_c,\, \sim] = \qr({\bm Y}_c)$.}
    \ENDIF
    \STATE{${[\hat{\bm Q}},\, {\hat{\bm R}}] = \qr({\bm \Omega}_c^* {\bm Q}_c)$.}
    \STATE{${\bm X} = {\hat{\bm R}}^{-1} {\hat{\bm Q}}^* {\bm Y}_r^*$.} 
    \STATE{$[\hat{\bm U}_p,\, {\bm \Lambda}_p,\, {\bm V}_p] = \tsvd({\bm X},\, p)$.}
    \STATE{${\bm U}_p = {\bm Q}_c \hat{\bm U}_p$.}
  \end{algorithmic}
\end{algorithm}

Incorporating this modification into the baseline 1-view method gives \cref{alg:1view}. Choosing $l_\text{c} = l_1$ gives the baseline algorithm as proposed by \cite{tropp2017}. In \cref{alg:1view} we use \textbf{a}) and \textbf{b}) to denote and group together steps that can be performed independently in parallel. \cref{alg:1view} includes an optional step, proposed by \cite{tropp2017}, that orthonormalizes the random test matrices ${\bm \Omega}_r$ and ${\bm \Omega}_c$. This option improves numerical accuracy of the 1-view method when using large oversampling \cite{tropp2017}, but this option was not used for the experiments considered here. Using $l_\text{c}$ and $l_1 = l_2$ can work just as well as and in some cases better than the schemes suggested by \cite{tropp2017}, see section \ref{sec:1viewresults}. For the test matrices we have considered, we have found that using $l_1 = l_2$ and $l_\text{c} \approx l_1 / 2$ is a good all-round parameter choice, similar to using \eqref{eq:overmed} and $l_\text{c} = l_1$.

Similar algorithms were presented by \cite{woolfe2008,halko2011structure,martinsson2016rsvd}; however, in the notation used here the methods presented by \cite{woolfe2008,tropp2017,martinsson2016rsvd} used $l_\text{c} = 0$ and $l_1 = l_2$. The algorithms presented here are more general in terms of the oversampling parameters and though the focus here is on choosing $l_1 = l_2$, the presented algorithms can use oversampling parameters fulfilling $0 \leq l_\text{c} \leq l_1 \leq l_2$. 

The algorithm proposed by \cite[Sect. 5.2]{woolfe2008} finds a low-rank approximation by solving a similar least-squares problem
\begin{equation} \label{eq:lstsqWoolfe}
({\bm \Omega}_c^* {\bm Q}_c) \hat{\bm X} = {\bm Y}_r^* {\bm Q}_r
,
\end{equation}
where ${\bm Q}_c$ and ${\bm Q}_r$ are formed, respectively, by the leading left-singular vectors of ${\bm Y}_c$ and ${\bm Y}_r$. Then the low-rank approximation is taken as ${\bm A} \approx {\bm Q}_c \llbracket \hat{\bm X} \rrbracket_p {\bm Q}_r^*$. Adopting aspects from \cref{alg:1view}, \cref{alg:1viewWoolfe} presents an extension of the Algorithm suggested by \cite[Sect. 5.2]{woolfe2008}. \cref{alg:1viewWoolfe} uses the same notation as \cref{alg:1view} to denote elements that can be evaluated independently in parallel. For instance, dealing with the matrix ${\bm Q}_r$, that does not appear in \cref{alg:1view}, does not necessarily increase computational time by much since it can be dealt with in parallel with other core tasks. For the matrices we considered, \cref{alg:1view} and \cref{alg:1viewWoolfe} performed similarly in terms of accuracy. For the tests in section \ref{sec:1viewresults} we present results using \cref{alg:1view} but omit showing results for \cref{alg:1viewWoolfe} as they are comparable. 

\begin{algorithm}[t!]
  \caption{Randomized 1-view method for TSVD: Woolfe variant.}\label{alg:1viewWoolfe}
  \begin{flushleft}
  \textbf{INPUT:} {Matrix ${\bm A} \in {\mathbb{R}}^{n_r \times n_c}$, integers $p > 0$ and $l_2 \geq l_1 \geq l_\text{c} \geq 0$}.\\
  \textbf{RETURNS:} {Approximate rank-$p$ SVD, ${\bm U}_p {\bm \Lambda}_p {\bm V}_p^*$, of ${\bm A}$.}
  \end{flushleft}
  \begin{algorithmic}[1]	
    \STATE{\textbf{a}) ${\bm \Omega}_r = \randn(n_c,\, p+l_1)$ and \textbf{b}) ${\bm \Omega}_c = \randn(n_r,\, p+l_2)$.}
    \STATE{Optional: \textbf{a}) ${\bm \Omega}_r =\text{orth}({\bm \Omega}_r)$ and \textbf{b}) ${\bm \Omega}_c =\text{orth}({\bm \Omega}_c)$.}
    \STATE{\textbf{a}) ${\bm Y}_c = {\bm A} {\bm \Omega}_r$   and   \textbf{b}) ${\bm Y}_r = {\bm A}^* {\bm \Omega}_c$.}
    \STATE{\textbf{a}) $[{\bm Q}_c,\, {\bm R}_c] = \qr({\bm Y}_c)$ and \textbf{b}) $[{\bm Q}_r,\, {\bm R}_r] = \qr({\bm Y}_r)$.}
    \STATE{\textbf{a}) $[\hat{\bm Q}_c,\, \sim,\, \sim] = \tsvd({\bm R}_c,\, p+l_\text{c})$ and \textbf{b}) $[\hat{\bm Q}_r,\, \sim,\, \sim] = \tsvd({\bm R}_r,\, p+l_\text{c})$.}
    	\STATE{\textbf{a}) ${\bm Q}_c = {\bm Q}_c \hat{\bm Q}_c$ and \textbf{b}) ${\bm Q}_r = {\bm Q}_r \hat{\bm Q}_r$.}
    \STATE{Least-squares solution: $({\bm \Omega}_c^* {\bm Q}_c) \hat{\bm X} = ({\bm Y}_r^* {\bm Q}_r)$.}
    \STATE{$[\hat{\bm U}_p,\, {\bm \Lambda}_p,\, \hat{\bm V}_p] = \tsvd(\hat{\bm X},\, p)$.}
    \STATE{\textbf{c}) ${\bm U}_p = {\bm Q}_c \hat{\bm U}_p$ and \textbf{d}) ${\bm V}_p = {\bm Q}_r \hat{\bm V}_p$.}
  \end{algorithmic}
\end{algorithm}

Following \cite{woolfe2008}, \cite{halko2011structure,martinsson2016rsvd} proposed a low-rank approximation which can also be written as ${\bm A} \approx {\bm Q}_c \llbracket \hat{\bm X} \rrbracket_p {\bm Q}_r^*$. The distinction is that \cite{halko2011structure,martinsson2016rsvd} suggested finding $\hat{\bm X}$ as the least-squares solution to \cref{eq:lstsqWoolfe} and a related equation $({\bm \Omega}_r^* {\bm Q}_r) \hat{\bm X}^* = {\bm Y}_c^* {\bm Q}_c$.
We are not aware of any research using the variant proposed by \cite{halko2011structure,martinsson2016rsvd}. However, we have found that for the matrices we have considered the \cite{halko2011structure,martinsson2016rsvd} variant mostly gave similar accuracy to the simpler algorithms presented above. Considering that result, and the additional computational complexity and expense of dealing with the additional least-squares equations, we do not consider the \cite{halko2011structure,martinsson2016rsvd} variant further.

\subsection{Oversampling when using $l_\text{c}$}

\subsubsection{Basic oversampling scheme}

Like the baseline algorithm, the modified algorithm requires some practical ways of choosing the oversampling parameters. The oversampling parameter $l_\text{c}$ can be chosen as
\begin{equation} \label{eq:lcutfixedratio} 
l_\text{c} = \lfloor \alpha l_1 \rfloor
,
\quad 
\text{with}
\quad 
0 \leq \alpha \leq 1
.
\end{equation}
Here we are concerned with schemes when $l_1 \simeq l_2$ or more precisely $l_1 = \lfloor (l_1 + l_2)/2 \rfloor = \lfloor T/2 \rfloor - p$. In that case, $l_\text{c}$ functions similarly in relation to $l_1$ as $l_1$ functions in relation to $l_2$ when using the baseline 1-view method. That is, choosing $l_\text{c}$ close to $l_1$ works well when the spectrum decays rapidly. Choosing $l_1 \simeq l_2$ and $\alpha = 1/2$ can work quite well for various matrices in a similar fashion to using \eqref{eq:overmed} with the baseline method.

\subsubsection{A minimum variance approach}

After all the effort of forming the sketches ${\bm Y}_c$ and ${\bm Y}_r$ we can try to apply some type of post-processing to determine an optimal $l_\text{c}$ parameter given the information contained in the sketches and sampling matrices. One simple idea could be to find the spectrum of ${\bm X}$ for some provisional $l_\text{c}$. Then if the provisional singular spectrum suggests, for instance, that the spectrum of ${\bm A}$ is likely to decay rapidly then a large $l_\text{c}$ can be chosen. However, the drawback is that it is unclear how this should be done, and it may require finding a complicated empirical selection scheme.

An approach we have found to work well for the problems we have considered looks at how the spectrum of ${\bm X}$ varies for all possible values of $l_\text{c}$ given ${\bm Y}_c$ and ${\bm Y}_r$. The approach is based on the observation, from multiple experiments, that the singular values of ${\bm X}$ tend to have large variance when choosing $l_\text{c}$ close to $l_1$ for $l_1 \simeq l_2$, because of the ill-conditioning issue. Furthermore, the spectrum of ${\bm X}$ can also exhibit large variance when choosing $l_\text{c}$ close to $0$ and often results in inferior approximations. We therefore consider finding $l_\text{c}$ that locally gives a matrix ${\bm X} (l_\text{c})$ with a singular spectrum that has minimum variance in some sense.

Defining the vector of locally normalized singular values ${\bm s} (l_\text{c})$, for $0 < l_\text{c} < l_1$ as
\begin{equation}
{\bm s} (l_\text{c}) :=
\left[ 
\frac{\lambda_1(l_\text{c}-1)}{\lambda_1(l_\text{c})}
,\, \cdots ,\,
\frac{\lambda_p(l_\text{c}-1)}{\lambda_p(l_\text{c})}
,\,
\frac{\lambda_1(l_\text{c})}{\lambda_1(l_\text{c})}
,\, \cdots ,\,
\frac{\lambda_p(l_\text{c})}{\lambda_p(l_\text{c})}
,\,
\frac{\lambda_1(l_\text{c}+1)}{\lambda_1(l_\text{c})}
,\, \cdots ,\,
\frac{\lambda_p(l_\text{c}+1)}{\lambda_p(l_\text{c})}
\right]
\end{equation}
and for $l_\text{c} = 0$
\begin{equation}
{\bm s} (l_\text{c}) :=
\left[ 
\frac{\lambda_1(l_\text{c})}{\lambda_1(l_\text{c})}
,\, \cdots ,\,
\frac{\lambda_p(l_\text{c})}{\lambda_p(l_\text{c})}
,\,
\frac{\lambda_1(l_\text{c}+1)}{\lambda_1(l_\text{c})}
,\, \cdots ,\,
\frac{\lambda_p(l_\text{c}+1)}{\lambda_p(l_\text{c})}
\right]
;
\end{equation}
the \textit{minimum variance} $l_\text{c}$ is found according to
\begin{equation} \label{eq:MinVar}
l_\text{c}^\text{MinVar} := \operatorname*{argmin}_{l_\text{c}} 
\,
\operatorname{Var}
[ {\bm s} (l_\text{c}) ]
,
\quad 
\text{with}
\quad 
0 \leq l_\text{c} < l_1
.
\end{equation}
Note that this post-processing does not require re-evaluating the sketches. We just need to generate the left singular vectors of ${\bm Y}_c$ and truncate them to $p + l_\text{c}$ for each trial $l_\text{c}$. For efficiency we evaluate ${\bm \Omega}_c^* {\bm Q}_c$ for $l_\text{c} = l_1$, where in this context ${\bm Q}_c$ denotes the left singular vectors of ${\bm Y}_c$. Then for each trial $l_\text{c}$ parameter we only need the first $p + l_\text{c}$ columns of ${\bm \Omega}_c^* {\bm Q}_c$.

The additional computational operations are $\mathcal{O}(l_1 [p + l_1]^2 n_r)$ ($\mathcal{O}([p + l_1]^2 n_r)$ for each possible $l_\text{c}$) when using \cref{alg:1view}. To improve efficiency it is possible to instead apply a variant of \cref{alg:1viewWoolfe} at a cost of $\mathcal{O}([p + l_1]^2 n_r + l_1 [p + l_1]^3)$. The $\mathcal{O}([p + l_1]^2 n_r)$ term comes from evaluating ${\bm Y}_r^* {\bm Q}_r$ once where ${\bm Q}_r$ is an orthonormal basis for ${\bm Y}_r$ (can skip truncating ${\bm Q}_r$). The remainder of the cost comes from solving the smaller least-squares problem $l_1 + 1$ times. This additional post-processing cost, for finding a better $l_\text{c}$, is potentially worth it for applications using the 1-view approach where the dominant computational cost comes from evaluating the matrix sketches. Furthermore, note that the computations pertaining to each trial $l_\text{c}$ can be carried out independently in parallel to the other trial variables.

Choosing $l_\text{c}$ by \eqref{eq:MinVar} worked well for the problems we considered and this minimum variance approach can perform quite close to the peak performance for a range of matrices, see section \ref{sec:1viewresults}. Nevertheless, the basic post-processing approach presented here to find a good $l_\text{c}$ parameter can be improved. For instance, when the rank $p$ is small (e.g., $p=1$) then the vector ${\bm s} (l_\text{c})$ presents a small sample size for estimating a variance. We have noticed that the minimum variance approach is less reliable for $p=1$. In that case it may be better to temporarily use a larger rank as input into the 1-view algorithm of choice and truncate afterwards, though this was not done for the test cases reported here.

\subsection{Extended sketching scheme} \label{sec:1viewextended}

In \cite{tropp2017}, an extended 1-view variant based on \cite{boutsidis2016,upadhyay2016} had the best performance for small memory budgets and matrices with flat spectra. However, for greater memory budgets and rapidly decaying spectra, \cite{tropp2017} found the baseline 1-view method to outperform the extended approach.

For the extended approach we assume that $l_1 = l_2$ and introduce additional subsampled randomized Fourier transform (SRFT) matrices
\begin{equation}
{\bm \Phi}_c \in  {\mathbb{R}}^{n_c \times s} ; \quad
{\bm \Phi}_r \in  {\mathbb{R}}^{n_r \times s} ,
\end{equation}
where $s \geq p + l_1$. Here, for $i \in \{ c, r \}$,
\begin{equation}
{\bm \Phi}_i = {\bm D}_i {\bm F}_i {\bm P}_i
\end{equation}
is a random SRFT matrix, where ${\bm D}_i \in {\mathbb{R}}^{n_i \times n_i}$ is a diagonal matrix with elements drawn independently from a Rademacher distribution (diagonal entries are $\pm 1$), ${\bm F}_i \in {\mathbb{R}}^{n_i \times n_i}$ denotes a discrete cosine transform matrix when ${\bm A}$ only has real entries or a discrete Fourier transform matrix when ${\bm A}$ has complex entries, and ${\bm P}_i \in {\mathbb{R}}^{n_i \times s}$ has $s$ columns that are sampled from the $n_i \times n_i$ identity matrix without replacement. Advantages of an SRFT sampling matrix are its low storage cost (only need to store $n_i + s$ entries for ${\bm D}_i$ and ${\bm P}_i$) and for a dense ${\bm A}$ then ${\bm A} {\bm \Phi}_c$ can be evaluated using $\mathcal{O}(n_r n_c \log{s})$ operations \cite{halko2011structure, liberty2007, woolfe2008} instead of the usual $\mathcal{O}(n_r n_c s)$ operations. However, we should note that we do not consider the extended approach discussed here as a viable option when approximating the types of Jacobian matrices motivating the present study. The reason is that for a Jacobian matrix, when using direct and adjoint methods to find ${\bm A}$ or its transpose times a matrix ${\bm \Phi}_i$, the cost scales with the number of columns $s$. The SRFT matrices need more columns than the standard Gaussian random matrices and using SRFT matrices is therefore not advantageous in the context of using adjoint and direct methods.

Using the extended approach we again find range and co-range sketches, but also an extended sketch ${\bm Z}$ as follows
\begin{equation}
{\bm Y}_c = {\bm A} {\bm \Omega}_r ; \quad
{\bm Y}_r = {\bm A}^* {\bm \Omega}_c ; \quad
{\bm Z} = {\bm \Phi}_r^* {\bm A} {\bm \Phi}_c .
\end{equation}
Subsequently, we evaluate the following QR-factorizations
\begin{align}
& {\bm Q}_c {\bm R}_c \vcentcolon = {\bm Y}_c ; \quad
{\bm Q}_r {\bm R}_r \vcentcolon = {\bm Y}_r ; \\
& {\bm U}_c {\bm T}_c \vcentcolon = {\bm \Phi}_r^* {\bm Q}_c ; \quad
{\bm U}_r {\bm T}_r \vcentcolon = {\bm \Phi}_c^* {\bm Q}_r .
\end{align}
Finally, a low-rank approximation can be found for the extended approach according to \cite{tropp2017}
\begin{equation} \label{eq:bwz}
\hat{\bm A}_\text{bwz} \vcentcolon = 
{\bm Q}_c {\bm T}_c^{\dagger} \left\llbracket {\bm U}_c^* {\bm Z} {\bm U}_r \right\rrbracket_p
({\bm T}_r^*)^{\dagger} {\bm Q}_r^*  .
\end{equation}
However, we have noticed that a more accurate low-rank approximation can be found by using instead
\begin{equation} \label{eq:bwzimproved}
\hat{\bm A}_\text{bwz2} \vcentcolon = 
{\bm Q}_c \left\llbracket {\bm T}_c^{\dagger} {\bm U}_c^* {\bm Z} {\bm U}_r
({\bm T}_r^*)^{\dagger} \right\rrbracket_p {\bm Q}_r^*  .
\end{equation}
Defining the QR-factorization $\hat{\bm Q} \hat{\bm R} := {\bm \Omega}_c^* {\bm Q}_c$ used in \cref{alg:1view}; the difference between \eqref{eq:bwz} and \eqref{eq:bwzimproved} is analogous to the difference between using $\hat{\bm A}_\text{woo} := {\bm Q}_c \hat{\bm R}^{\dagger} \llbracket {\hat{\bm Q}}^* {\bm Y}_r^*  \rrbracket_p$, and the baseline 1-view method $\hat{\bm A}_\text{tropp} := {\bm Q}_c \llbracket \hat{\bm R}^{\dagger}  {\hat{\bm Q}}^* {\bm Y}_r^*  \rrbracket_p$. The approximation $\hat{\bm A}_\text{woo}$ is an adaptation made by \cite{tropp2017} for an approach given by \cite{woodruff2014}. As previously mentioned, \cite{tropp2017} found that the $\hat{\bm A}_\text{woo}$ variant did not perform nearly as well as the baseline $\hat{\bm A}_\text{tropp}$.

As mentioned by \cite{tropp2017} it is unclear how to choose the oversampling factors for the above extended approach. Storing the sampling matrices and sketches requires $(2p + l_1 + l_2 + 1)(n_r + n_c) + s(s + 2)$ numbers \cite{tropp2017}. Tropp et al. \cite{tropp2017} compared the extended approach to the baseline method by considering oversampling parameters such that $(2p + l_1 + l_2 + 1)(n_r + n_c) + s(s + 2)$ approximately equals the memory allocated to the simpler 1-view methods $T (n_r + n_c)$. We do the same when comparing the extended approach with \cref{alg:1view}.

For the matrices tested here, we found that using \eqref{eq:bwzimproved} with
\begin{equation} \label{eq:bwzafac8}
l_1 = l_2 = 0.8 \lfloor T/2 - p \rfloor
\end{equation}
and choosing $s$ as the largest value fulfilling $(2p + l_1 + l_2 + 1)(n_r + n_c) + s(s + 2) \approx T (n_r + n_c)$ resulted in approximation errors close to the best performance. In terms of memory, this approach performed the best out of the 1-view schemes mentioned above, see section \ref{sec:1viewresults}.

\subsection{Matrices streamed row wise} \label{sec:singlepassRowMajor}

For the special case where the elements of a large matrix are read into random-access memory a few rows at a time, it is possible to attain the accuracy of a 2-view method while only accessing the elements of the matrix once \cite{yu2017single,yu2018singlepower}. Looping over all the rows in ${\bm A}$ only once the matrices ${\bm Y}_c = {\bm A} {\bm \Omega}_r$ and $\hat{\bm Y}_r = {\bm A}^* {\bm Y}_c$ can be found according to ${\bm Y}_c [i,\, :] = {\bm A}[i,\, :] {\bm \Omega}_r$ and $\hat{\bm Y}_r = {\bm A}^* {\bm Y}_c = \sum_i ({\bm A}[i,\, :])^* {\bm Y}_c [i,\, :]$, where ${\bm A}[i,\, :]$ denotes the $i$th row of ${\bm A}$. Evaluating $[{\bm Q}_c ,\, {\bm R}_c] = \qr ({\bm Y}_c)$ and finding the least-squares solution ${\bm B}$ to ${\bm B}^* {\bm R}_c = \hat{\bm Y}_r$, then we have a QB-factorization such that ${\bm Q}_c {\bm B} \approx {\bm A}$. In \cite{yu2017single}, this QB-factorization was found by an iterative blocked algorithm instead of forming ${\bm Y}_c := {\bm Q}_c {\bm R}_c$ and solving for ${\bm B}$ directly.

If ${\bm R}_c$ is invertible then ${\bm B}^* = \hat{\bm Y}_r {\bm R}_c^{-1} = {\bm A}^* {\bm Q}_c$ and we get the basic randomized 2-view approximation in a single pass over the rows of ${\bm A}$. However, it is worth noting that the least-squares problem ${\bm B}^* {\bm R}_c = \hat{\bm Y}_r$ can be badly conditioned. The conditioning issue can, for instance, be dealt with by finding a TSVD of ${\bm R}_c$ and estimating a solution using a pseudoinverse.

The above 1-view procedure is advantageous for large matrices that are stored out-of-core in row-major or column-major format. Of course, when the matrix is stored in column major format and accessed column wise, then the above procedure would be applied to the matrix's transpose \cite{yu2017single}. However, this approach does not apply to general streaming matrices, nor the Jacobian matrices motivating this study.

In \cite{yu2018singlepower}, the above single-pass ideas were extended to the power iteration approach to obtain higher quality approximations of large matrices accessed row wise. The following section discusses generalizing a pass-efficient randomized block Krylov algorithm and combining randomized block Krylov methods with the ideas of \cite{yu2017single,yu2018singlepower}.

\section{Pass-efficient randomized block Krylov methods} \label{sec:blockKrylovMethods}

As mentioned previously, if the 1-view approach is not suitably accurate, then more accurate methods such as \cref{alg:gensubit} should be considered. While \cref{alg:gensubit} is more pass efficient than classical iterative methods, which involve matrix vector multiplication, its pass efficiency can be improved by using a randomized block Krylov approach \cite{drineas2017,halko2011pca,martinsson2010normalized,musco2015krylov,rokhlin2009}. The next subsection gives a brief overview of the current state-of-the art for randomized block Krylov methods. The subsections that follow, outline two extensions of the block Krylov approach. The first approach, like \cref{alg:gensubit}, enables the user to specify any budget of views $v \geq 2$. The second approach incorporates the ideas discussed in section \ref{sec:singlepassRowMajor} to give a more pass-efficient method for large matrices stored in row-major format.

\subsection{A prototype block Krylov method}

Randomized methods were presented by
\cite{rokhlin2009,martinsson2010normalized,halko2011pca} which combine power or subspace iteration with a block Krylov approach, for improved accuracy given a number of power iterations. The applicability of this type of randomized block Krylov approach was demonstrated by \cite{halko2011pca} for the task of approximating data sets that are too large to fit into random-access memory. Therefore, the data needed to be stored out-of-core and accessing a matrix was time-consuming. Likewise, \cite{musco2015krylov} demonstrated the advantages of using a randomized block Krylov method for large data sets. Furthermore, their analysis suggests that the block Krylov approach can achieve a desired accuracy using fewer matrix views than a standard randomized subspace iteration method and using less time. Recently, \cite{drineas2017} presented a theoretical analysis of the randomized block Krylov method for the problem of approximating the subspace spanned by the dominant left-singular vectors of a matrix.

Following the work of
\cite{drineas2017,halko2011pca,martinsson2010normalized,musco2015krylov,rokhlin2009}, an approximate basis ${\bm Q}_c$ for the range of ${\bm A}$ can be found as the basis of a Krylov subspace given by
\begin{equation}
{\bm K}_c = \range (
{\bm A} {\bm \Omega}_r \,\,\,\,
[{\bm A} {\bm A}^*] {\bm A} {\bm \Omega}_r \,\,\,\,
\cdots \,\,\,\,
[{\bm A} {\bm A}^*]^q {\bm A} {\bm \Omega}_r
) ,
\end{equation}
where the random sampling matrix ${\bm \Omega}_r \in \mathbb{R}^{n_c \times (p+l)}$ is defined as in the discussion for the subspace iteration method. A low-rank approximation can be generated by ${\bm Q}_c \llbracket {\bm Q}_c^* {\bm A} \rrbracket_p$. Possible improvements over the simpler subspace iteration method come from using a larger and more complete basis ${\bm Q}_c$. This is the prototype algorithm used by \cite{drineas2017,halko2011pca,martinsson2010normalized,musco2015krylov,rokhlin2009}. For numerical stability it is best to evaluate $[{\bm A} {\bm A}^*]^j {\bm A} {\bm \Omega}_r$ using subspace iteration \cite{martinsson2010normalized,musco2015krylov}.

\subsection{A more general block Krylov method for $v \geq 2$}

Like the standard subspace iteration method, the above prototype block Krylov method uses an even number of $2(q+1)$ views. Again, we can generalize the above prototype approach to work for any views $v \geq 2$. For an odd number of views we simply form a basis ${\bm Q}_r$ for the co-range using a Krylov subspace defined as
\begin{equation} \label{eq:krylovCoRange}
{\bm K}_r = \range (
[{\bm A}^* {\bm A}] {\bm \Omega}_r \,\,\,\,
[{\bm A}^* {\bm A}]^2 {\bm \Omega}_r \,\,\,\,
\cdots \,\,\,\,
[{\bm A}^* {\bm A}]^{(v-1)/2} {\bm \Omega}_r
) .
\end{equation}

Combining the above thoughts we arrive at \cref{alg:blkKrylov}, which is a more flexible block Krylov method than that used by \cite{drineas2017,halko2011pca,martinsson2010normalized,musco2015krylov,rokhlin2009} and allows $v \geq 2$. For an even number of views \cref{alg:blkKrylov} uses ${\bm A} \approx {\bm Q}_c {\bm Q}_c^* {\bm A}$ otherwise it uses ${\bm A} \approx {\bm A} {\bm Q}_r {\bm Q}_r^*$. When the goal is to approximate the left singular vectors of a matrix ${\bm J}$ the same rule of thumb regarding accuracy applies to \cref{alg:blkKrylov} as to \cref{alg:gensubit}: apply \cref{alg:blkKrylov} to ${\bm J}^*$ for even $v$ but ${\bm J}$ for odd $v$. This is the opposite to how the block Krylov approach was used by \cite{drineas2017,musco2015krylov}.

\begin{algorithm}
\caption{Randomized SVD using generalized block Krylov method.} \label{alg:blkKrylov}
\begin{flushleft}
\textbf{INPUT:} Matrix ${\bm A} \in {\mathbb{R}}^{n_r \times n_c}$, integers $p > 0$, $l \geq 0$ and $v \geq 2$.\\
\textbf{RETURNS:} Approximate rank-$p$ SVD, ${\bm U}_p {\bm \Lambda}_p {\bm V}_p^*$, of ${\bm A}$.
\end{flushleft}
\begin{algorithmic}[1]
\STATE{${\bm Q}_r^0 = \randn(n_c,\, p+l)$.}
\FOR{$j = 1$ \textbf{to} $v-1$}
	\IF{$j$ is odd}
        \STATE{\textbf{if} $j<v-1$ \textbf{then} $[{\bm Q}_c^j,\, \sim] = \qr({\bm A} {\bm Q}_r^{j-1})$; \textbf{else} ${\bm Q}_c^j= {\bm A} {\bm Q}_r^{j-1}$.}
    \ELSE
        \STATE{\textbf{if} $j<v-1$ \textbf{then}  $[{\bm Q}_r^j,\, \sim] = \qr({\bm A}^* {\bm Q}_c^{j-1})$; \textbf{else} ${\bm Q}_r^j= {\bm A}^* {\bm Q}_c^{j-1}$.}
     \ENDIF
\ENDFOR
\IF{$v$ is even}
	\STATE{${\bm K}_c = [{\bm Q}_c^1 \,\,\,\, {\bm Q}_c^3 \,\,\,\, \cdots \,\,\,\, {\bm Q}_c^{v-1}]$.}
    \STATE{$[{\bm Q}_c,\, \sim] = \qr({\bm K}_c)$.}
    \STATE{$[{\bm Q}_r,\, {\bm R}_r] = \qr({\bm A}^* {\bm Q}_c)$.}
	\STATE{$[\hat{{\bm V}}_p,\, {\bm \Lambda}_p,\, \hat{\bm U}_p] = \tsvd({\bm R}_r,\, p)$.}
\ELSE
	\STATE{${\bm K}_r = [{\bm Q}_r^2 \,\,\,\, {\bm Q}_r^4 \,\,\,\, \cdots \,\,\,\, {\bm Q}_r^{v-1}]$.}
    \STATE{$[{\bm Q}_r,\, \sim] = \qr({\bm K}_r)$.}
    \STATE{$[{\bm Q}_c,\, {\bm R}_c] = \qr({\bm A} {\bm Q}_r)$.}
	\STATE{$[\hat{{\bm U}}_p,\, {\bm \Lambda}_p,\, \hat{\bm V}_p] = \tsvd({\bm R}_c,\, p)$.}
\ENDIF
\STATE{${\bm U}_p = {\bm Q}_c \hat{\bm U}_p$ and ${\bm V}_p = {\bm Q}_r \hat{\bm V}_p$.}
\end{algorithmic}
\end{algorithm}

Notice that \cref{alg:blkKrylov} is the same as \cref{alg:gensubit} for $v < 4$. The main computational cost involved in \cref{alg:blkKrylov} is from multiplying ${\bm A}$ with $(v + \lfloor v/2 \rfloor - 1)(p+l)$ vectors and the cost of the QR-factorizations is $\mathcal{O} (v^2 n_c [p+l]^2)$, assuming $n_r = n_c$. This is a higher computational cost for a given $v$ than that of \cref{alg:gensubit}, which requires multiplying ${\bm A}$ with $v (p+l)$ vectors and the cost of the QR-factorizations is $\mathcal{O} (v n_c [p+l]^2)$. However, \cref{alg:blkKrylov} may require fewer views than \cref{alg:gensubit} to achieve a desired accuracy when \cref{alg:gensubit} needs more than four views.

For $l=0$ and an even number of views, \cite{musco2015krylov} showed that \cref{alg:blkKrylov} requires $\mathcal{O} (1/\sqrt{\epsilon})$ views to achieve $\norm{{\bm A} - {\bm Q}_c {\bm Q}_c^* {\bm A}} \leq (1 + \epsilon) \norm{{\bm A} - \llbracket {\bm A} \rrbracket_p}$; however, \cref{alg:gensubit} requires $\mathcal{O} (1/\epsilon)$ views \cite{musco2015krylov}. Since their analysis assumed no oversampling it is unclear how the methods compare in practice when oversampling is applied. By oversampling a bit more the accuracy of \cref{alg:gensubit} can be increased and the oversampling can be chosen such that the cost of \cref{alg:gensubit} is similar to the cost of \cref{alg:blkKrylov}. Which approach should be chosen in practice?

The results in section \ref{sec:resultsBlockKrylov} suggest that \cref{alg:gensubit} with additional oversampling can be competitive with \cref{alg:blkKrylov} when the matrix has reasonably rapid spectral decay. \cref{alg:blkKrylov}, however, is more advantageous for problems where the tail of the singular spectrum is flat. This can be the case when dealing with noisy data matrices, which was the problem motivating \cite{musco2015krylov}.

\subsection{A block Krylov method for matrices stored row wise}

The ideas discussed in section \ref{sec:singlepassRowMajor} can also be applied in the block Krylov framework to improve pass-efficiency when dealing with large matrices that are accessed row by row. In that case, ${\bm A}^* {\bm A}$ times a matrix can be evaluated in a single-pass. Therefore, the Krylov subspace for the co-range of ${\bm A}$ given in \eqref{eq:krylovCoRange} can be evaluated using half the number of views needed by \cref{alg:blkKrylov}. Using this approach could be advantageous when dealing with large data matrices stored out-of-core.

\section{Experimental results} \label{sec:results}

\subsection{Test matrices}

For testing the randomized low-rank approximation methods discussed in the present study we consider seven test matrices. The first is a $6,135$ by $24,000$ Jacobian matrix ${\bm S}_\text{D}$ arising from a synthetic geothermal inverse problem like the one examined in \cite{bjarkason2017randomized}. The leading, normalized singular values of ${\bm S}_\text{D}$ are depicted in \cref{fig:SingularValues}. The other test matrices are the same as some of those used by \cite{tropp2017}; two matrices that have a singular spectrum with a flat tail, two where the trailing singular values decay polynomially and two where the decay of the tail is exponential.

\begin{figure}[t!]
  \centering
\includegraphics[width=0.8\textwidth]{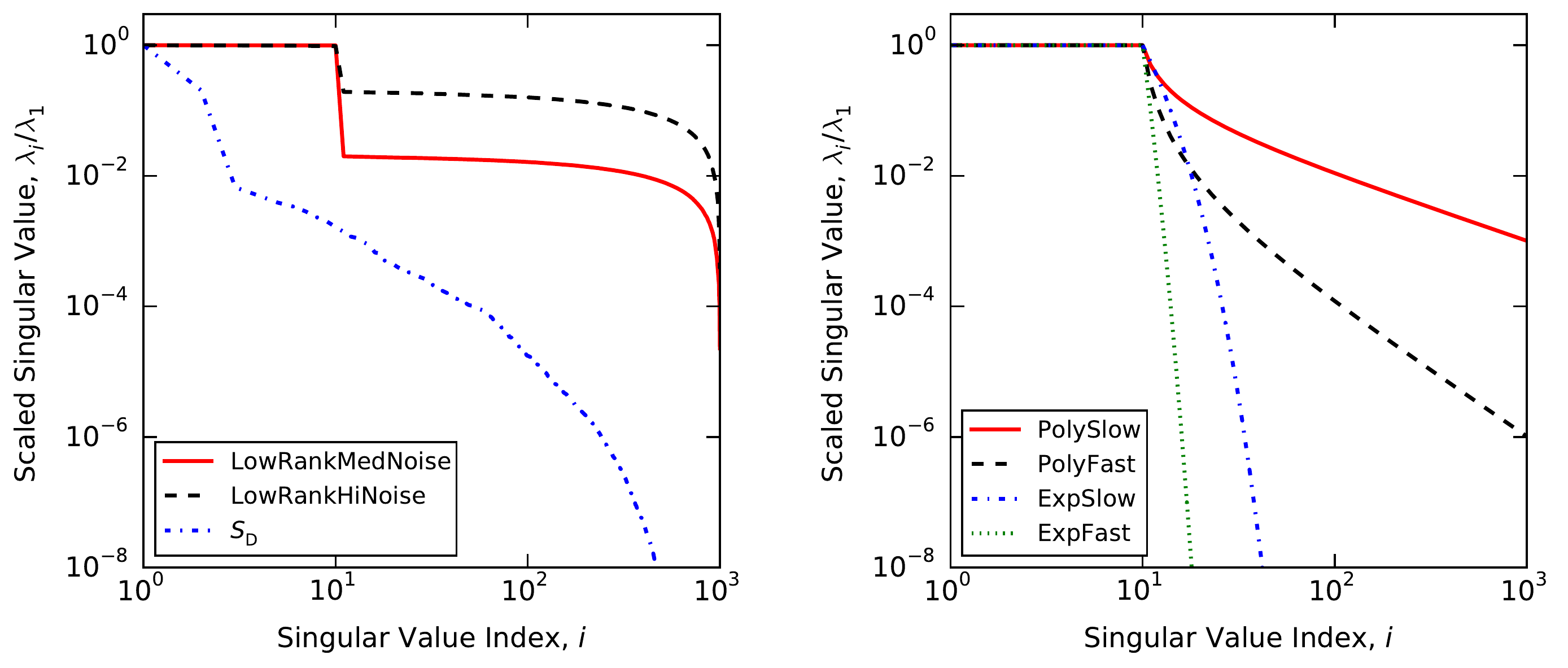}
  \caption{Scaled singular values for the test matrices.}
  \label{fig:SingularValues}
\end{figure}

Defining $R=10$ and $n=1,000$, each of the flat matrices is a single realization of a matrix which is a combination of a matrix of rank $R$ and a random noise matrix. That is, these flat or low-rank plus noise matrices are $n \times n$ matrices given by
\begin{equation}
\diag(\underbrace{1,\, \cdots ,\, 1}_\textrm{R},\, 0,\, \cdots ,\, 0)
+
\sqrt{ 
\frac{\eta R}{2 n^2}
}
\left( {\bm G} + {\bm G}^* \right) ,
\end{equation}
where the elements of ${\bm G} \in {\mathbb{R}}^{n \times n}$ are drawn independently from a standard Gaussian distribution. The noisy test matrices are a medium noise matrix (\texttt{LowRankMedNoise}) with $\eta = 10^{-2}$ and a high noise matrix (\texttt{LowRankHiNoise}) with $\eta = 1$.

The polynomial decay matrices are diagonal matrices of the following type:
\begin{equation}
\diag ( \underbrace{1,\, \cdots ,\, 1}_\textrm{R},\, 
2^{-\rho},\, 3^{-\rho},\, \cdots ,\, [n - R +1]^{-\rho} ) .
\end{equation}
The first polynomially decaying matrix has a slow decay $\rho = 1$ (\texttt{PolySlow}); the other decays faster with $\rho = 2$ (\texttt{PolyFast}). Similarly, the exponentially decaying matrices are defined as
\begin{equation}
\diag ( \underbrace{1,\, \cdots ,\, 1}_\textrm{R},\, 
10^{- \theta},\, 10^{- 2 \theta},\, \cdots ,\, 10^{- (n - R) \theta} ) .
\end{equation}
We use one matrix with a rapid decay $\theta = 1$ (\texttt{ExpFast}) and another with a slower decay $\theta = 0.25$ (\texttt{ExpSlow}). The singular spectra of the above test matrices are shown in \cref{fig:SingularValues}.

\subsection{Quantifying algorithmic accuracy}

To quantify the performance of the randomized algorithms, the generated rank-$p$ approximations $\hat{\bm A}_\text{out}$ are compared with the optimal rank-$p$ factorization of the matrix of interest ${\bm A}$. In terms of the Frobenius and spectral norms the optimal rank-$p$ factorization is the rank-$p$ TSVD of ${\bm A}$, denoted by $\llbracket {\bm A} \rrbracket_p$.

For some of the tests, we follow \cite{tropp2017} and compare the relative Frobenius norm errors, defined as
\begin{equation} \label{eq:relfroberr}
\text{Relative Frobenius Error } :=
\frac{ \norm{ {\bm A} - \hat{\bm A}_\text{out} }_\text{F} }{ \norm{ {\bm A} - \llbracket {\bm A} \rrbracket_p }_\text{F} } - 1 .
\end{equation}
Likewise, we also look at the relative spectral norm errors, defined as
\begin{equation} \label{eq:relspecerr}
\text{Relative Spectral Error } :=
\frac{ \norm{ {\bm A} - \hat{\bm A}_\text{out} } }{ \norm{ {\bm A} - \llbracket {\bm A} \rrbracket_p } } - 1 .
\end{equation}

\subsection{Effectiveness of the generalized subspace iteration method}

\subsubsection{Standard application}

\begin{figure}[b!]
  \centering
\includegraphics[width=\textwidth]{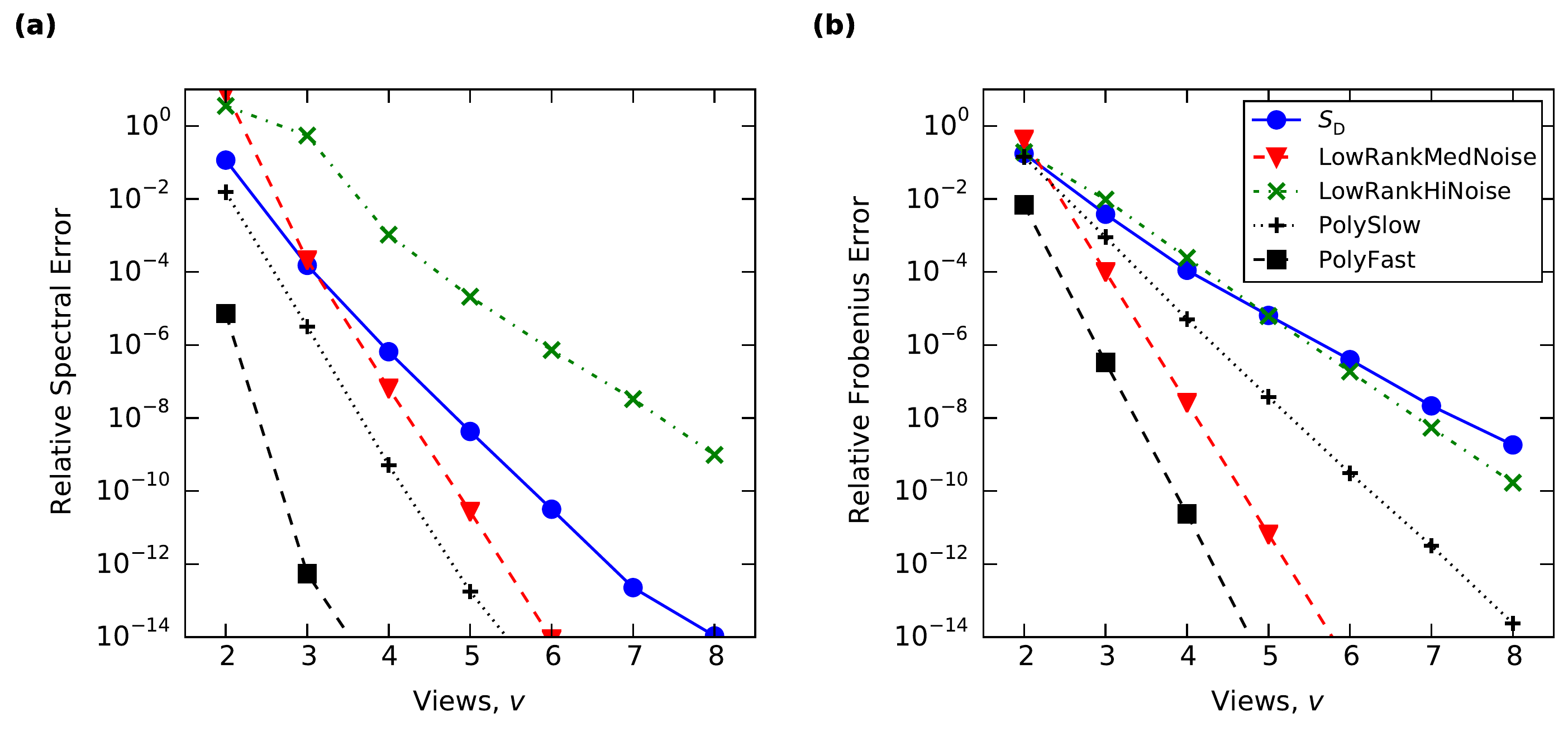}
  \caption{Matrix approximation errors when using the generalized subspace iteration method (\cref{alg:gensubit}), with a target rank $p = 10$ and oversampling parameter $l = 10$. a) Relative spectral norm error \eqref{eq:relspecerr}. b) Relative Frobenius norm error \eqref{eq:relfroberr}.}
  \label{fig:GenSubSpecFrobErrs}
\end{figure}

\cref{fig:GenSubSpecFrobErrs} shows the decay of the spectral and Frobenius norm errors as functions of matrix views when using the generalized subspace iteration method (\cref{alg:gensubit}). To generate the results shown in \cref{fig:GenSubSpecFrobErrs} we examined generating approximations of the test matrices with a target rank $p = 10$. The oversampling was fixed as $l = 10$. The approximation errors given in \cref{fig:GenSubSpecFrobErrs} are averages over $50$ calls of \cref{alg:gensubit}. As we would expect, based on \cref{thm:powit,thm:genpowit}, the relative errors decay exponentially with the number of matrix views. Approximation errors for the exponentially decaying test matrices are not shown in \cref{fig:GenSubSpecFrobErrs} because they have low approximation errors for any number of views.

The advantage of using the more general subspace \cref{alg:gensubit} over the standard subspace \cref{alg:subit} (same as \cref{alg:gensubit} for even views) is apparent from looking at \cref{fig:GenSubSpecFrobErrs}. That is, for some matrices an odd number of views can suffice to achieve a desired accuracy. In that case, the additional matrix view needed when using \cref{alg:subit} is excessive and a waste of computational effort.

\subsubsection{Application to normal matrices} \label{sec:resultsnormalmatrix}

In section \ref{sec:apply2normalmat} it was suggested that using ${\bm J}^*$ as input into \cref{alg:gensubit} may, for some applications, perform differently to using ${\bm J}$ as input. Here we consider using \cref{alg:gensubit} to form a low-rank approximation of ${\bm J}^* {\bm J}$. Here, the low-rank approximations are compared against $\llbracket {\bm J}^* {\bm J} \rrbracket_p$ and expected errors were estimated by running each method $500$ times. 

\begin{figure}[h!]
  \centering
\includegraphics[width=\textwidth]{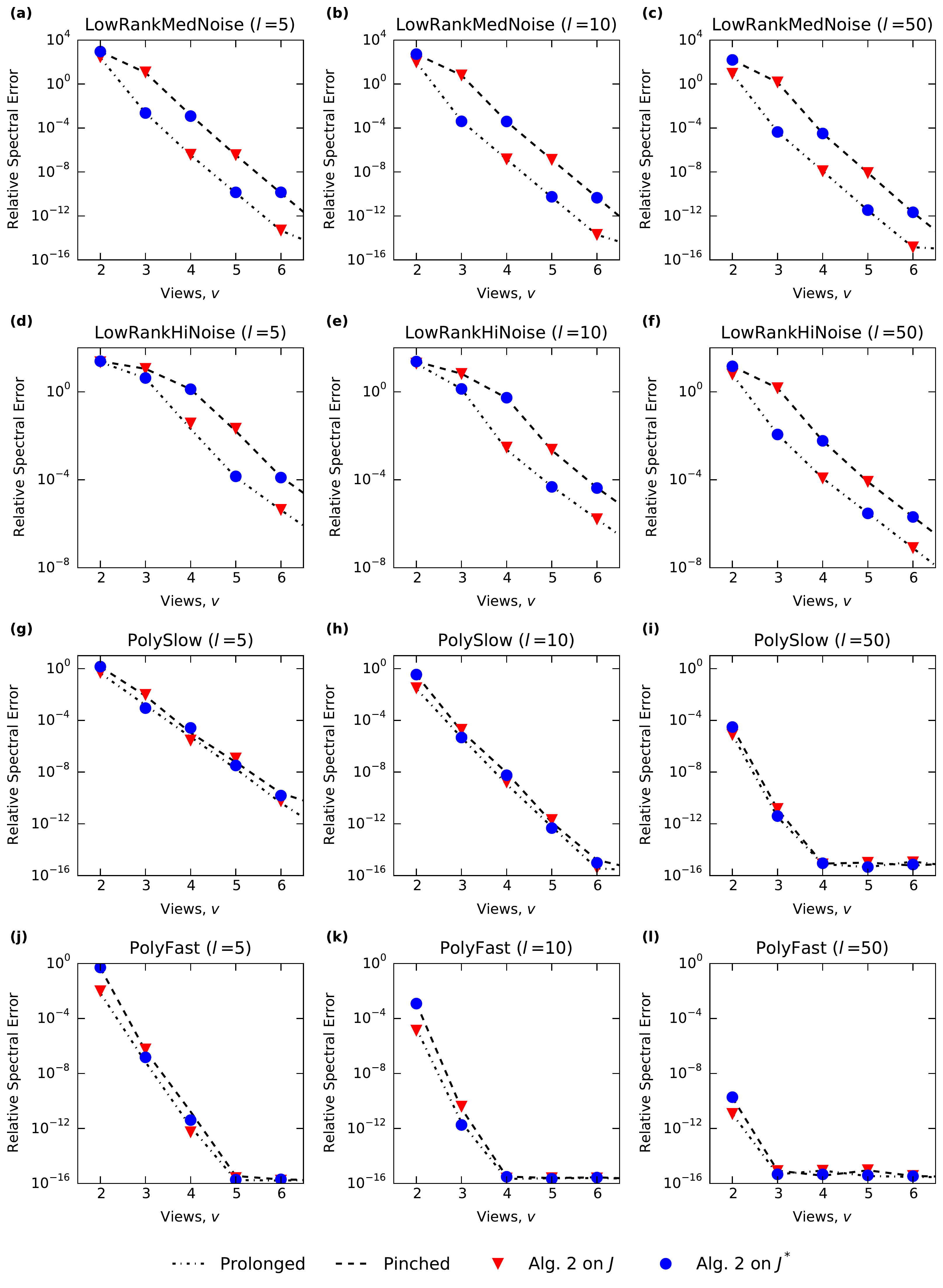}
  \caption{For selected test matrices ${\bm J}$, spectral norm approximation errors \eqref{eq:relspecerr} of the normal matrix ${\bm J}^* {\bm J}$ when using the Prolonged (Nystr\"{o}m) sketch \cref{alg:nystrom}, the Pinched sketch \cref{alg:pinched}, ${\bm J}$ as input into \cref{alg:gensubit}, or ${\bm J}^*$ as input into \cref{alg:gensubit}. For all the plots the target rank was $p=10$.}
  \label{fig:SpectralErrorNormalFig1}
\end{figure}

\cref{fig:SpectralErrorNormalFig1} showcases the difference between using ${\bm J}$ and ${\bm J}^*$ as input into \cref{alg:gensubit}, for the noisy and polynomial test matrices. As suggested by the discussion in section \ref{sec:apply2normalmat}, using ${\bm J}$ as input for even views and ${\bm J}^*$ as input for odd views gives the best accuracy. This is algebraically equivalent to using the Nystr\"{o}m type approach (\cref{alg:nystrom}). The alternative choices for \cref{alg:gensubit}, which equate to using the pinched method (\cref{alg:pinched}), perform worse. Notice that for some problems a poor choice of input into \cref{alg:gensubit} can result in little gain in accuracy over using the same input matrix and one less view. However, for some problems the difference can be less extreme or even negligible.

\cref{fig:SpectralErrorNormalFig2} tells a similar story when considering the Jacobian test matrix ${\bm S}_\text{D}$. For the results in \cref{fig:SpectralErrorNormalFig2}, we used a diagonal matrix containing the top $1,000$ singular values of ${\bm S}_\text{D}$ as a surrogate for ${\bm S}_\text{D}$. This was done to speed up the analysis but does not affect the overall results. The results show for ${\bm S}_\text{D}$ that the Nystr\"{o}m template has a noticeable advantage for a small number of views ($v < 4$), especially when considering the 2-view variants. This is good to know since four views or less is what we would hope to be sufficient.

\begin{figure}[htbp]
  \centering
\includegraphics[width=\textwidth]{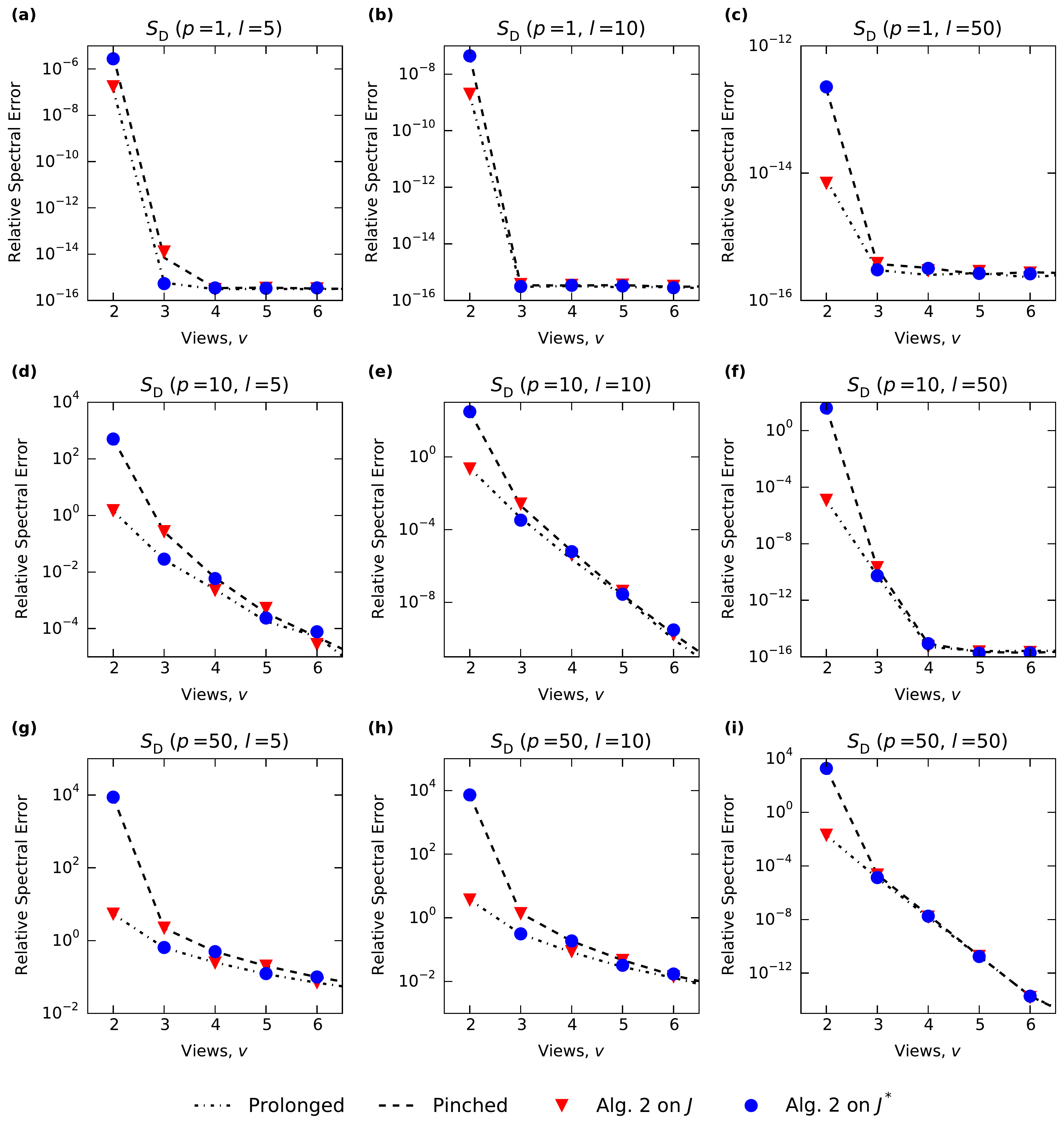}
  \caption{For the $6,135$ by $24,000$ Jacobian test matrix ${\bm J} = {\bm S}_\text{D}$, spectral norm approximation errors \eqref{eq:relspecerr} of the normal matrix ${\bm J}^* {\bm J}$ when using the Prolonged (Nystr\"{o}m) sketch \cref{alg:nystrom}, the Pinched sketch \cref{alg:pinched}, ${\bm J}$ as input into \cref{alg:gensubit}, or ${\bm J}^*$ as input into \cref{alg:gensubit}.}
  \label{fig:SpectralErrorNormalFig2}
\end{figure}

\subsubsection{Subspace iteration compared with the block Krylov approach}\label{sec:resultsBlockKrylov}

We also consider differences between applying the block Krylov \cref{alg:blkKrylov} and the simpler subspace iteration \cref{alg:gensubit}. For these tests we considered the Jacobian test matrix and the three matrices with the slowest decay, since the other test matrices do not require many views to achieve a high-quality factorization. \cref{fig:BlkKrylovError} compares the pass efficiency of \cref{alg:blkKrylov} and \cref{alg:gensubit} when using $l=10$ and averaging over $50$ runs. 

\begin{figure}[b!]
  \centering
\includegraphics[width=\textwidth]{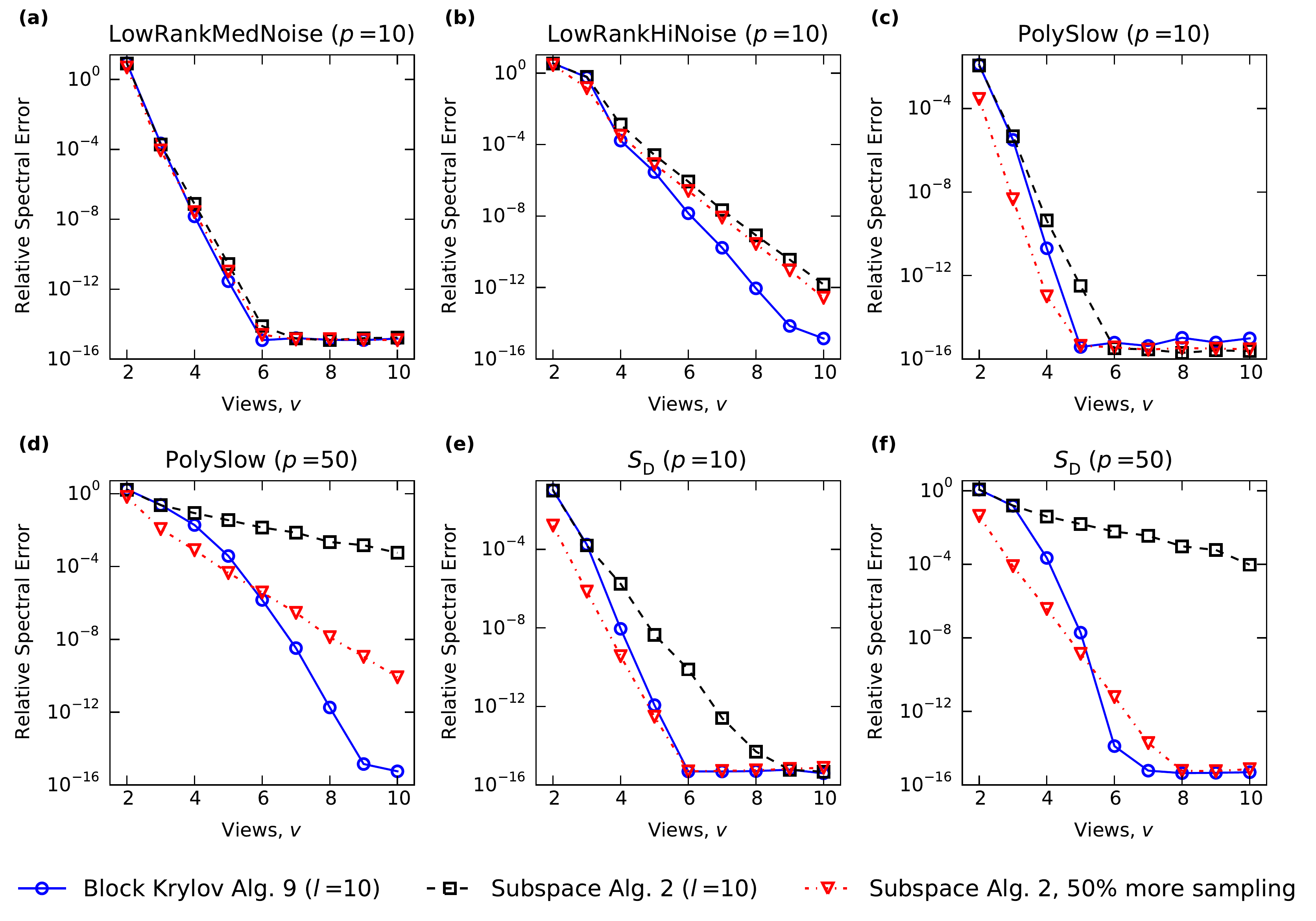}
  \caption{Improvement in matrix approximation errors when increasing the number of matrix passes or views, when using the generalized subspace iteration method (\cref{alg:gensubit}) and the generalized block Krylov algorithm (\cref{alg:blkKrylov}).}
  \label{fig:BlkKrylovError}
\end{figure}

\begin{figure}[htbp]
  \centering
\includegraphics[width=\textwidth]{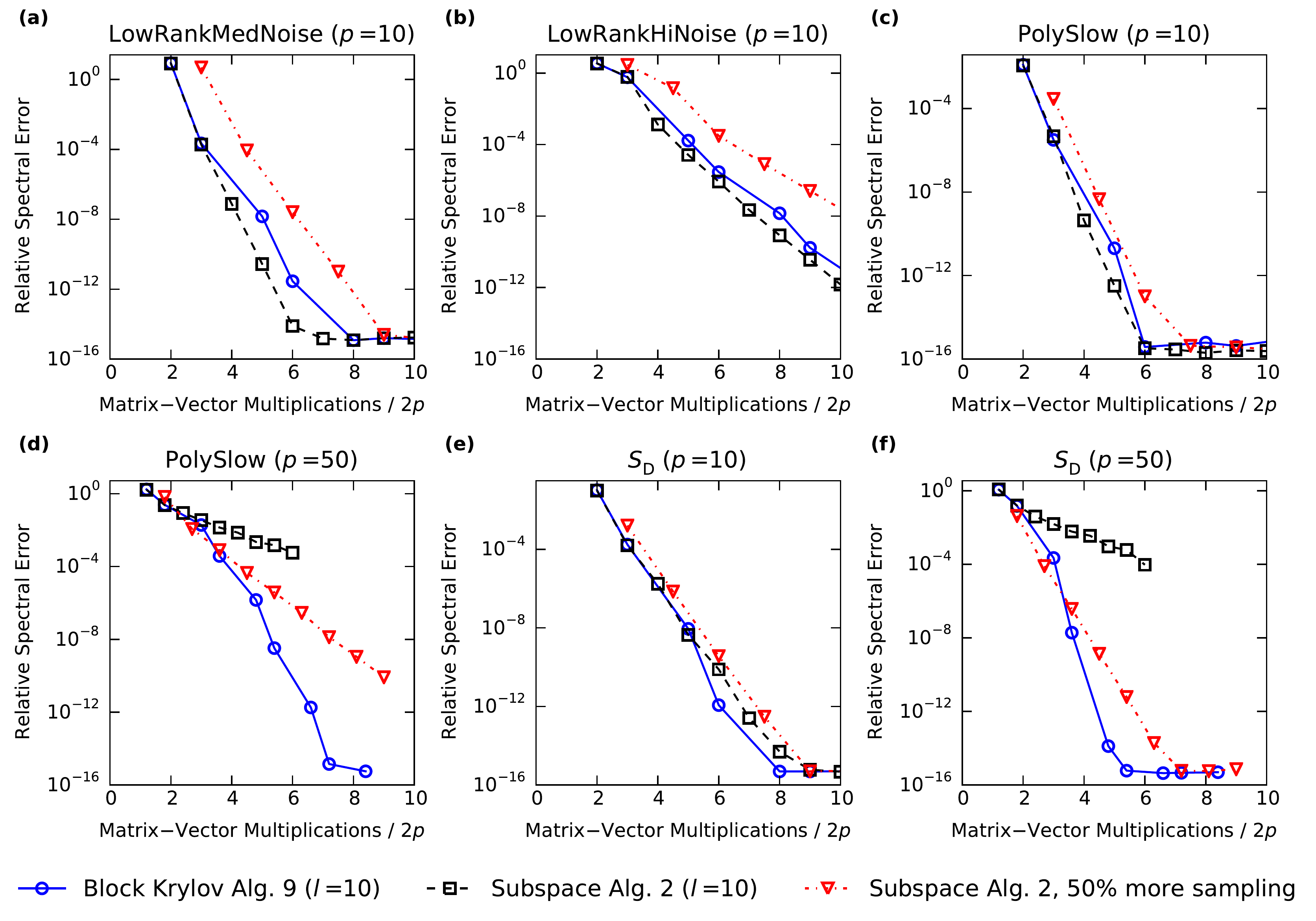}
  \caption{Improvement in matrix approximation errors with respect to the number of matrix-vector multiplications, when using the generalized subspace iteration method (\cref{alg:gensubit}) and the generalized block Krylov algorithm (\cref{alg:blkKrylov}).}
  \label{fig:BlkKrylovMatrixMultCost}
\end{figure}

The results show for fixed oversampling and views that the block Krylov method is more accurate. However, Figures \ref{fig:BlkKrylovError}c)-f) also show that the subspace iteration method can be made more accurate by oversampling more. For the matrices that have reasonably rapid spectral decay, additional oversampling can make the pass efficiency of \cref{alg:gensubit} comparable to \cref{alg:blkKrylov}. Figures \ref{fig:BlkKrylovError} and \ref{fig:BlkKrylovMatrixMultCost} show for reasonably decaying matrices that \cref{alg:gensubit} can achieve accuracy akin to \cref{alg:blkKrylov} using a similar number of matrix-vector multiplications and views. This is important to know since the main computational expense for the Jacobian matrices motivating this study is from the number of matrix-vector multiplications, which are found using adjoint and direct simulations.

Considering these results and the fact that \cref{alg:blkKrylov} has higher computational cost associated with the large Krylov space, \cref{alg:gensubit} may be advantageous for some problems. Furthermore, for each view \cref{alg:gensubit} involves multiplying the input matrix with a matrix having a fixed number of columns. This may be more suitable when dealing with adjoint and direct simulations as the sample size $p+l$ could be tailored to the available computational hardware. \cref{alg:blkKrylov}, on the other hand, would involve a variable number of adjoint and direct solves, which could make it harder to optimize. However, for data matrices with trailing singular values that decay slowly, the block Krylov approach can be better, see \cref{fig:BlkKrylovError,fig:BlkKrylovMatrixMultCost}, and \cite{musco2015krylov}. Furthermore, \cref{fig:BlkKrylovError,fig:BlkKrylovMatrixMultCost} show that increased sampling may not be worth the cost when the decay is slow.

\subsection{Comparison of 1-view streaming methods} \label{sec:1viewresults}

Section \ref{sec:1view} gives templates for 1-view variants that potentially improve upon the state-of-the-art 1-view methods outlined in \cite{tropp2017}. Here we are mainly interested in studying \cref{alg:1view} when using $l_1 \simeq l_2$. By $l_1 \simeq l_2$ we mean $l_1 =  \lfloor T/2 \rfloor - p$ (which gives $l_1 = l_2$ or $l_1 = l_2 - 1$). Section \ref{sec:testlcut} compares different schemes for the modification parameter $l_\text{c}$ when using \cref{alg:1view} with $l_1 \simeq l_2$. Section \ref{sec:testvsbaseline} compares selected schemes from section \ref{sec:testlcut} with the baseline 1-view scheme recommended by \cite{tropp2017}. Finally, section \ref{sec:optimal1viewtests} considers the improved extended scheme introduced in section \ref{sec:1viewextended} and how well practical oversampling schemes can perform compared with the best possible. The relative errors presented in this section are averages over $50$ trials.

\subsubsection{Choosing $l_\text{c}$}\label{sec:testlcut}

\cref{fig:FrobErrLcut} compares, for $l_1 \simeq l_2$, schemes where $l_\text{c}$ is chosen a priori as a fixed ratio of $l_1$ \eqref{eq:lcutfixedratio} and the minimum variance approach \eqref{eq:MinVar}. The minimum variance approach uses post-processing to choose $l_\text{c}$, every time a low-rank approximation is generated. Therefore, $l_\text{c}$ can vary between runs of the minimum variance approach though other input parameters stay fixed.

To increase accuracy $l_\text{c}$ should take on larger values when the singular spectrum decays fast. But, notice from \cref{fig:FrobErrLcut} how choosing a large value, especially $l_\text{c} = l_1$, is a bad idea unless the spectral decay is rapid and $l_1$ is large enough. These results for the relationship between $l_\text{c}$ and $l_1$ are analogous to the results for the relationship between $l_1$ and $l_2$ when using the baseline 1-view method \cite{tropp2017}. From \cref{fig:FrobErrLcut} we draw the conclusion that the simple choice $l_\text{c} = \lfloor l_1/2 \rfloor$ works quite well overall. However, like the baseline oversampling schemes proposed by \cite{tropp2017} and discussed in section \ref{sec:baselineoversampling}, the fixed $l_\text{c}$ schemes require some prior ideas about the expected spectral decay to get peak performance. The adaptive minimum variance approach, on the other hand, performs close to the best for most of the test problems and oversampling budgets.

\begin{figure}[b!]
  \centering
\includegraphics[width=0.95\textwidth]{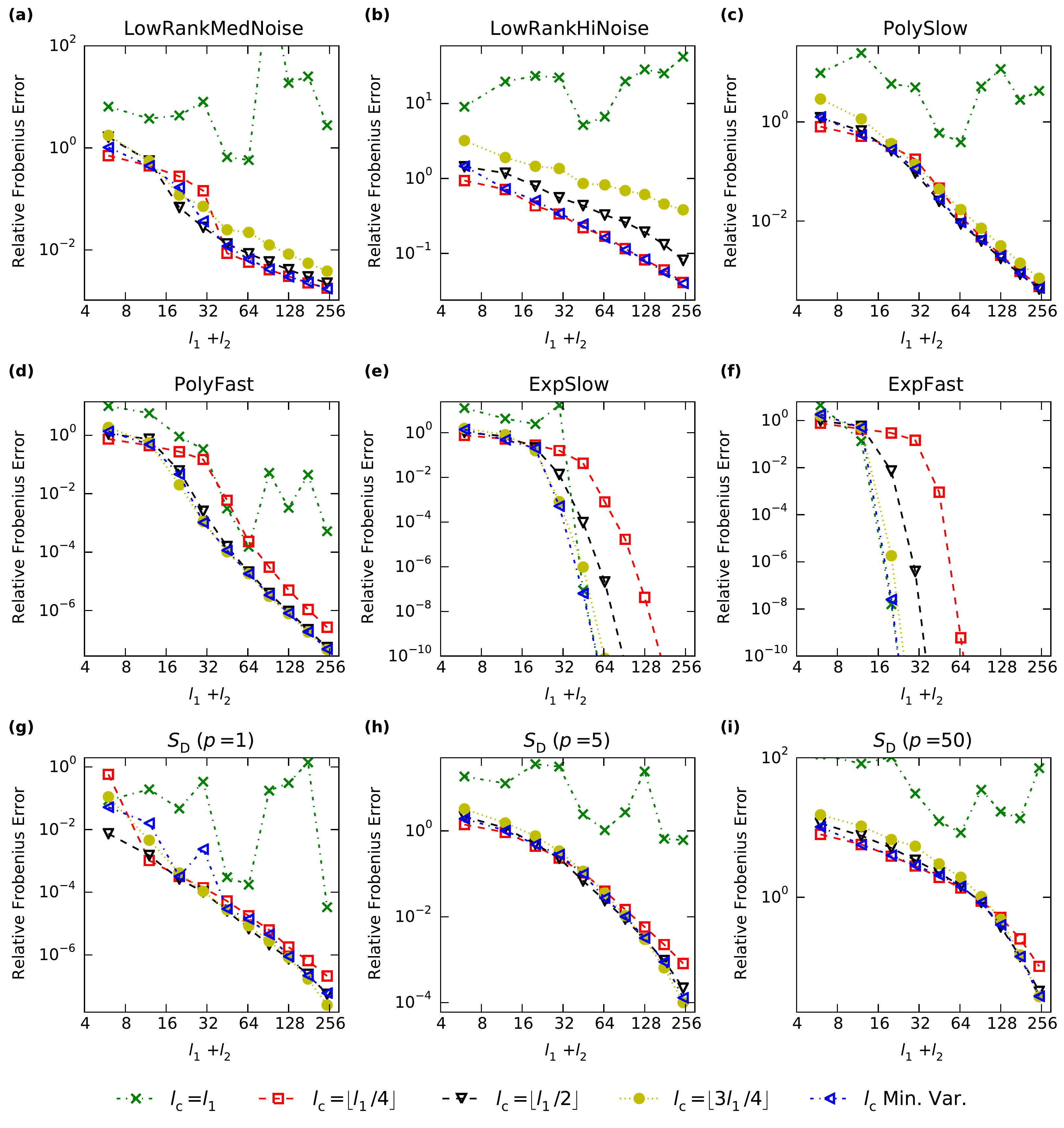}
  \caption{Relative Frobenius errors when using \cref{alg:1view} with $l_1 \simeq l_2$ ($l_1 = \lfloor (l_1 + l_2)/2 \rfloor = \lfloor T/2 \rfloor - p$). The rank of the approximations is $p=5$, unless specified otherwise. The plots compare the performance of choosing a fixed $l_\text{c}$ using \eqref{eq:lcutfixedratio} with choosing $l_\text{c}$ adaptively according to the minimum variance (Min. Var.) approach \eqref{eq:MinVar}.}
  \label{fig:FrobErrLcut}
\end{figure}

\subsubsection{Baseline method compared with $l_1 \simeq l_2$ variants} \label{sec:testvsbaseline}

\cref{fig:FrobErrTroppVsLcut} compares the baseline 1-view approach with \cref{alg:1view} using $l_1 \simeq l_2$, with $l_\text{c}$ chosen using the minimum variance approach or $l_\text{c} = \lfloor l_1/2 \rfloor$. For the baseline method we chose to run three variants, one for each oversampling scheme proposed by \cite{tropp2017}. These baseline oversampling schemes are designed for matrices with flat (\texttt{FLAT}: \eqref{eq:overflat}), moderately decaying (\texttt{DECAY}: \eqref{eq:overmed}) and rapidly decaying (\texttt{RAPID}: \eqref{eq:overrapid}) singular spectra. \cref{fig:FrobErrTroppVsLcut} shows that \cref{alg:1view} with $l_1 \simeq l_2$ performs similarly to the baseline 1-view approaches. The scheme using $l_\text{c} = \lfloor l_1/2 \rfloor$ works similarly to the most widely applicable baseline scheme (\texttt{DECAY}: \eqref{eq:overmed}). Again, the minimum variance approach worked the best overall.

\begin{figure}[t!]
  \centering
\includegraphics[width=0.95\textwidth]{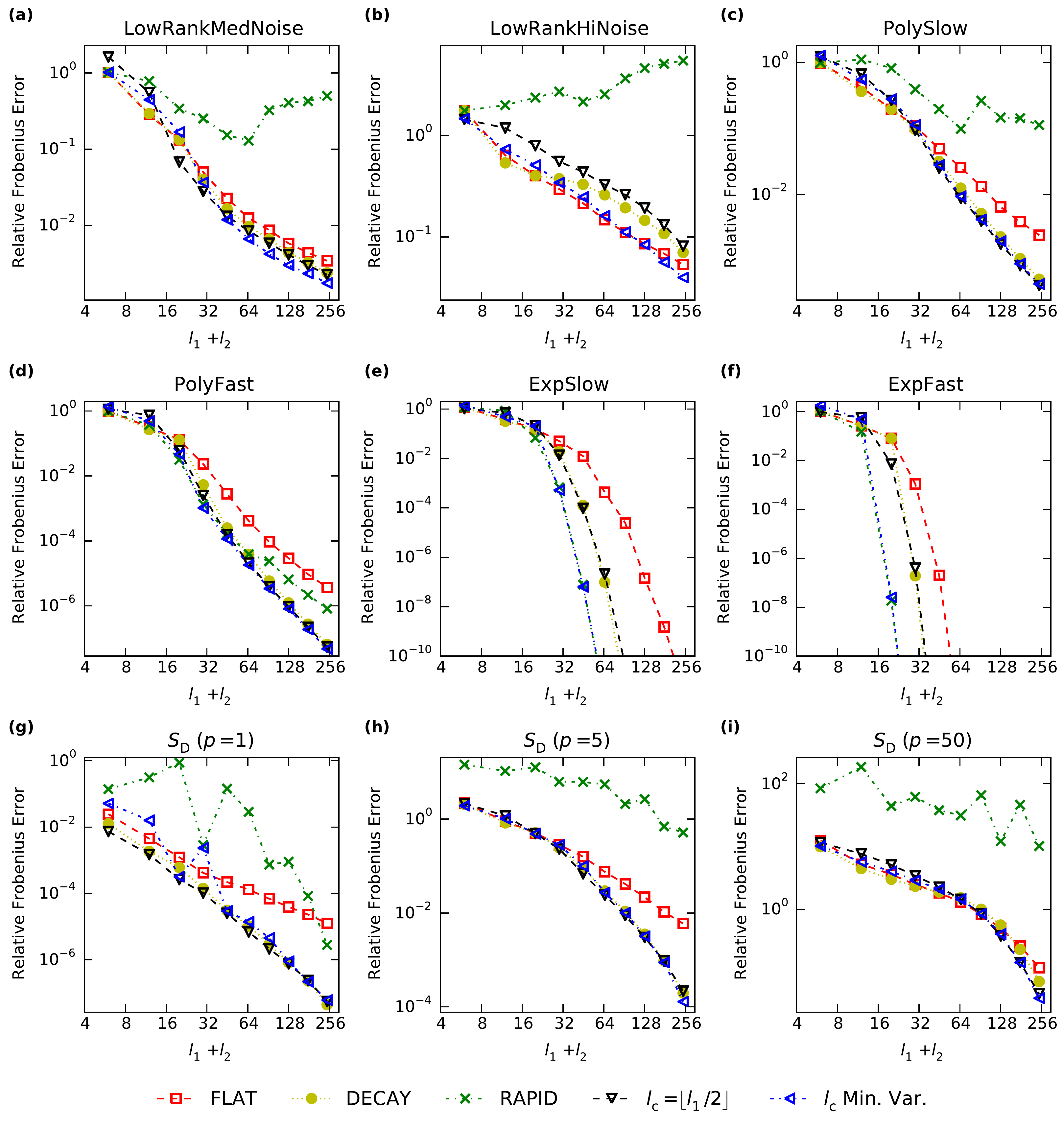}
  \caption{Relative Frobenius errors when using \cref{alg:1view} compared against the baseline 1-view algorithm (see section \ref{sec:1viewbaselinealg}). The rank of the approximations is $p=5$, unless specified otherwise. For \cref{alg:1view} we consider $l_1 \simeq l_2$ with $l_\text{c}$ chosen as $\lfloor l_1/2 \rfloor$ or chosen adaptively by the minimum variance (Min. Var.) approach \eqref{eq:MinVar}. For the baseline method we consider oversampling schemes for flat (\texttt{FLAT}: \eqref{eq:overflat}), moderately decaying (\texttt{DECAY}: \eqref{eq:overmed}), and rapidly decaying (\texttt{RAPID}: \eqref{eq:overrapid}) singular spectra.}
  \label{fig:FrobErrTroppVsLcut}
\end{figure}

\begin{figure}[t!]
  \centering
\includegraphics[width=0.95\textwidth]{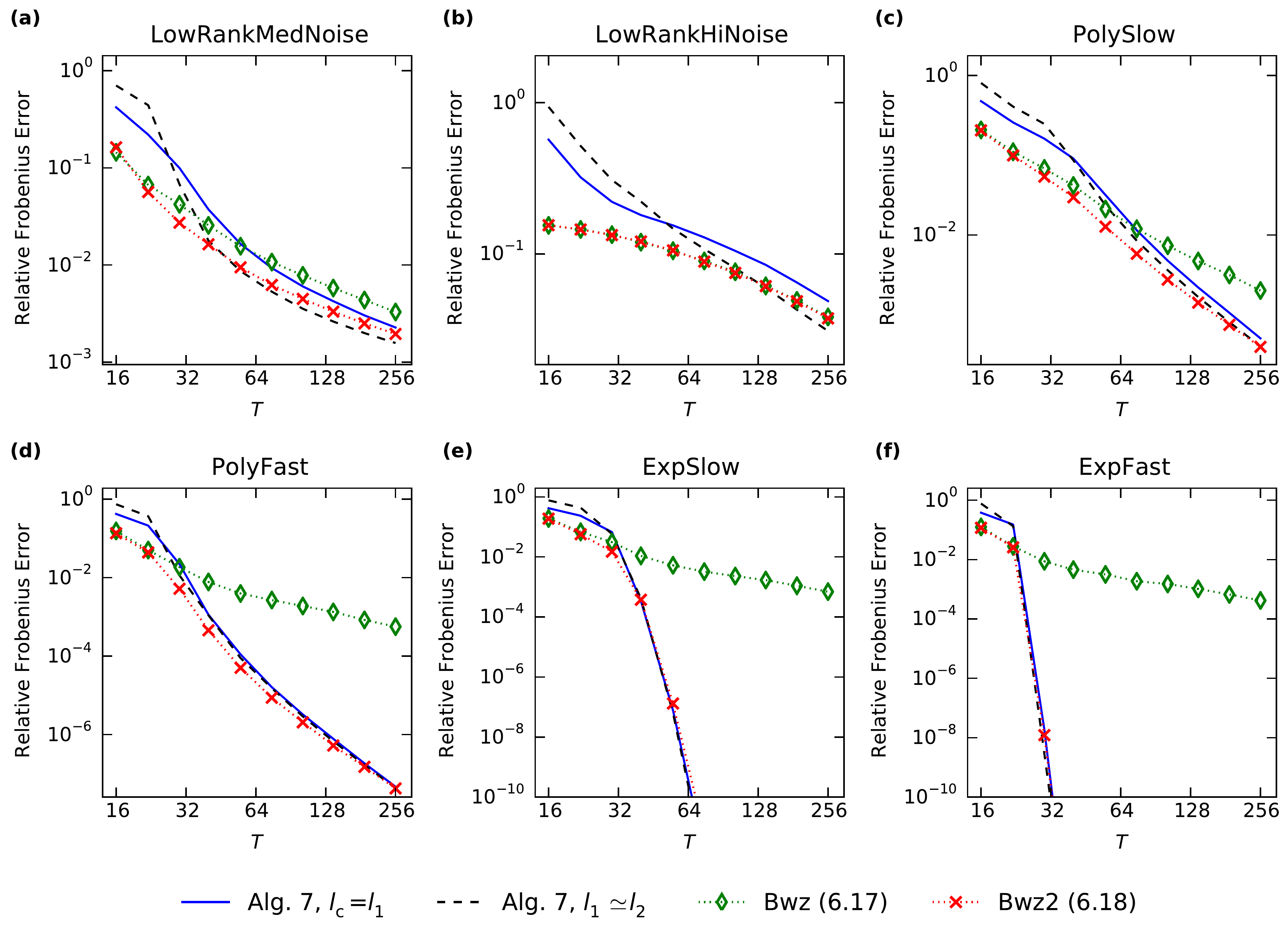}
  \caption{Best performance of \cref{alg:1view} (with $l_1 \simeq l_2$), the baseline 1-view algorithm (see section \ref{sec:1viewbaselinealg}), an inaccurate extended 1-view method \eqref{eq:bwz}, and an improved extended 1-view method \eqref{eq:bwzimproved}. The methods are compared in terms of their performance with respect to the memory budget $T$. All tests used a fixed rank $p=5$ for the approximations.}
  \label{fig:Oracle5}
\end{figure}

\subsubsection{Optimal performance and the extended scheme} \label{sec:optimal1viewtests}

To compare further the 1-view methods, outlined in section \ref{sec:1view}, we look at their optimal performance using a comparison similar to that used in \cite{tropp2017}. More specifically, we consider the baseline 1-view method, \cref{alg:1view} with $l_1 \simeq l_2$ and the two extended sketching variants. Here we try to estimate the \textit{best performance} a method can be expected to achieve for a given sampling budget $T$. Here the best performance of a method was found by running each method $50$ times for feasible oversampling parameters given the storage budget $T$ and choosing the fixed oversampling scheme that gave the minimum average Frobenius error. This differs slightly from the \textit{oracle performance} measure used by \cite{tropp2017}, where they ran each method for feasible oversampling parameters and found the minimum error. They repeated this and took the average minimum error over $20$ trials. The oracle performance measure gives lower errors than the best performance measure. In our opinion, the best performance measure better reflects the peak performance that we could hope to achieve. It should be noted that \cite{tropp2017} did not claim that oracle performance is achievable in practice but used it as a tool for comparing the possible upside of various methods.

\cref{fig:Oracle5} compares the best performance for the baseline method, \cref{alg:1view} with $l_1 \simeq l_2$, and the extended sketch methods. As claimed in section \ref{sec:1viewextended} the proposed modification to the extended sketch \cref{eq:bwzimproved} performs substantially better than using \cref{eq:bwz}. Using \cref{eq:bwzimproved} also retains the good performance characteristics of using \cref{eq:bwz} for small oversampling and flat spectra. Overall for a fixed memory budget $T$ the improved extended sketching method \cref{eq:bwzimproved} performs the best out of the presented 1-view methods. The baseline 1-view method and \cref{alg:1view} ($l_1 \simeq l_2$) have similar peak performance characteristics. But how do practical implementations compare to the best performance?

\cref{fig:OracleVsPractical5} shows approximation errors that are achievable in practice using the minimum variance approach along with the best performance for \cref{alg:1view} with $l_1 \simeq l_2$. The minimum variance scheme does well and gives approximations comparable to the best performance. \cref{fig:OracleVsPractical5} also depicts results of using \cref{eq:bwzimproved} with oversampling parameters chosen according to \cref{eq:bwzafac8}. This scheme, arrived at by simple trial and error, results in approximation errors that are close to the best performance. The oversampling scheme \cref{eq:bwzafac8} and the minimum variance method worked well for the matrices we tested, but this may not necessarily be the case in general. However, the presented 1-view methods and oversampling schemes are promising.

\begin{figure}[t!]
  \centering
\includegraphics[width=0.95\textwidth]{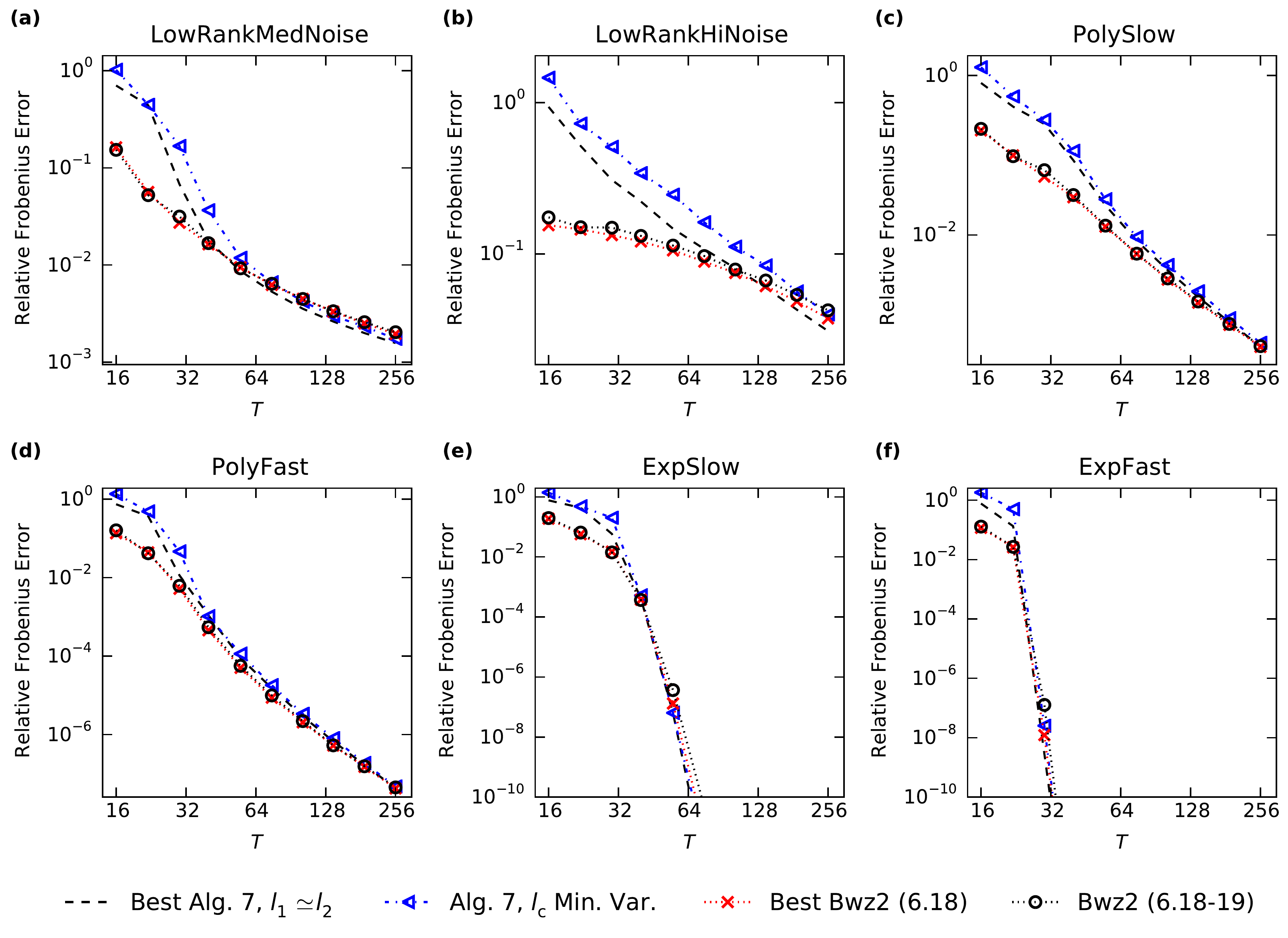}
  \caption{Best performance of \cref{alg:1view} (with $l_1 \simeq l_2$) and an improved extended 1-view method \eqref{eq:bwzimproved} compared with practical versions of both methods. The methods are compared in terms of their performance with respect to the memory budget $T$. All tests used a fixed rank $p=5$ for the approximations. \texttt{Min. Var.} refers to using \cref{alg:1view} with \eqref{eq:MinVar}.}
  \label{fig:OracleVsPractical5}
\end{figure}

\section{Conclusions}
\label{sec:conclusions}

We have presented practical randomized algorithms which are variations of popular randomized algorithms for a rank-$p$ approximation of a rectangular matrix. The study was aimed at practical algorithms that can be applied to a range of problems which may have an arbitrary budget for the number of times the matrix is accessed or viewed. To this end, we considered a more general randomized subspace iteration algorithm that generates a rank-$p$ factorization using any number of views $v \geq 2$ as a substitute for a widely used subspace algorithm which only uses an even number of views. We also extended this approach to a block Krylov framework, which can be advantageous for matrices that have a heavy tail in their singular spectrum.

For applications that can only afford viewing the data matrix once, we presented improved randomized 1-view methods. Adapting a 1-view algorithm recently recommended by \cite{tropp2017}, we arrived at a simple algorithm and sampling scheme that can outperform the algorithm recommended by \cite{tropp2017} for a range of matrices. The 1-view approach proposed here achieves this by using an adaptive post-processing step which analyses in a simple way the information contained in the randomized sketches created by the 1-view scheme. We also improve an extended 1-view method which was tested by \cite{tropp2017}. A simple adjustment to the extended 1-view method makes it considerably more accurate and it can outperform the previously mentioned 1-view method when the goal is to minimize the memory required by the randomized matrix sketches. Furthermore, based on 1-view and power iteration methods proposed by \cite{yu2017single,yu2018singlepower}, we proposed a highly pass-efficient randomized block Krylov algorithm that is suited to approximating large matrices stored out-of-core in either row-major or column-major format.

We also discussed how randomized algorithms can be applied to the type of nonlinear and large-scale inverse problems, motivating this study. In particular, we discussed various randomized strategies for approximating Levenberg-Marquardt model updates at a low computational cost. In relation to the motivating problem we discussed subtle differences, accuracy wise, between applying the randomized subspace or block-Krylov algorithms to a matrix of interest or its transpose. The subtle differences may be important when estimating normal matrices such as those appearing in inverse problems and are important in applications that involve estimating either the left-singular or right-singular subspace of a matrix. Computational experiments supported the claims made on the properties of the presented randomized algorithms.

\appendix
\section{Proof of \cref{thm:genpowit}}\label{sec:proof}
To prove \cref{thm:genpowit} we use the following theorem, which is analogous to \cite[Thm. 9.2]{halko2011structure}.

\begin{theorem}[Half power scheme]\label{thm:halfpowaid}
Let ${\bm A} \in {\mathbb{R}}^{n_r \times n_c}$ and let ${\bm \Omega}_r \in {\mathbb{R}}^{n_c \times s}$. For an integer $q \geq 1$ set ${\bm Y}_r = ({\bm A}^* {\bm A})^q {\bm \Omega}_r$. Let ${\bm Q}_r \in {\mathbb{R}}^{n_c \times s}$ be an orthonormal matrix which forms a basis for the range of ${\bm Y}_r$. Then
\begin{displaymath}
\norm{ {\bm A} ({\bm I} - {\bm Q}_r {\bm Q}_r^*) }^{2q}
\leq
\norm{ ({\bm A}^* {\bm A})^q ({\bm I} - {\bm Q}_r {\bm Q}_r^*) } \, .
\end{displaymath}
\end{theorem}
\begin{proof}
Given an orthogonal projector ${\bm R}$, a nonnegative diagonal matrix ${\bm D}$ and $t \geq 1$ we use the claim that \cite[Prop. 8.6]{halko2011structure}
\begin{equation}\label{eq:halkoclaim}
\norm{{\bm R} {\bm D} {\bm R}}^t 
\leq
\norm{{\bm R} {\bm D}^t {\bm R}} \, .
\end{equation}
Now, with the SVD ${\bm A} = {\bm U} {\bm \Lambda} {\bm V}^*$ and an orthogonal projector ${\bm P} \in {\mathbb{R}}^{n_c \times n_c}$ we get
\begin{align*}
\norm{{\bm A} {\bm P}}^{4q} &=
\norm{{\bm P} {\bm A}^* {\bm A} {\bm P}}^{2q} =
\norm{({\bm V}^* {\bm P} {\bm V}) {\bm \Lambda}^2 ({\bm V}^* {\bm P} {\bm V})}^{2q} \\
& \leq 
\norm{({\bm V}^* {\bm P} {\bm V}) {\bm \Lambda}^{4q} ({\bm V}^* {\bm P} {\bm V})} = 
\norm{{\bm P} ({\bm A}^* {\bm A})^{2q} {\bm P}} \\ 
&= 
\norm{{\bm P} ({\bm A}^* {\bm A})^{q} ({\bm A}^* {\bm A})^{q} {\bm P}} =  
\norm{({\bm A}^* {\bm A})^{q} {\bm P}}^2 \, .
\end{align*}
The second and fourth relations apply because of the unitary invariance of the spectral norm. The third relation applies because ${\bm V}^* {\bm P} {\bm V}$ is an orthogonal projector and \cref{eq:halkoclaim} therefore applies. Finally, replace ${\bm P}$ with the orthogonal projector ${\bm I} - {\bm Q}_r {\bm Q}_r^*$ and take the square root of the first and last spectral norms.
\end{proof}

With \cref{thm:halfpowaid} in hand we prove \cref{thm:genpowit} in the same fashion as \cite{halko2011structure} proved \cref{thm:powit}.

\begin{proof}[Proof of \cref{thm:genpowit}]
Define ${\bm B} = ({\bm A}^* {\bm A})^q$.
Then, by H\"older's inequality and \cref{thm:halfpowaid},
\begin{displaymath}
\mathbb{E} \left[ \norm{ {\bm A} - {\bm A} {\bm Q}_r {\bm Q}_r^* } \right] 
\leq
\left( \mathbb{E} \left[ \norm{ {\bm A} - {\bm A} {\bm Q}_r {\bm Q}_r^* }^{2q}  \right] \right)^{1/(2q)} 
\leq
\left( \mathbb{E} \left[ \norm{ {\bm B} - {\bm B} {\bm Q}_r {\bm Q}_r^* } \right]  \right)^{1/(2q)} 
.
\end{displaymath}
Now,
\begin{align*}
 \mathbb{E} \left[ \norm{ {\bm B} - {\bm B} {\bm Q}_r {\bm Q}_r^* } \right] 
=
\mathbb{E} \left[ \norm{ {\bm B}^* - {\bm Q}_r {\bm Q}_r^* {\bm B}^* } \right]
 \leq
\left( 1 + \sqrt{\frac{p}{l-1}} \right) \lambda_{p+1}^{2q} +
\frac{\mathrm{e} \sqrt{p+l}}{l} 
\left( \sum_{j>p} \lambda_j^{4q} \right)^{1/2} \, .
\end{align*}
The above uses that ${\bm B}$ has singular values $\left\{ \lambda_j^{2q} \right\}_{j=1}^{\text{min}(n_r,n_c)}$ and the inequality follows from \cite[Thm. 10.6]{halko2011structure}, which is a special case of \cref{thm:powit} with $q=0$.
\end{proof}

\section{Properties of truncated Levenberg-Marquardt updates} \label{sec:appendixTSVDLMupdates}
Consider the Levenberg-Marquardt update equation \eqref{eq:LMupdate} and that we have a TSVD approximation of the Jacobian ${\bm J}\approx \hat{\bm J}$. We assume that the rank of the approximation is $s$, that is ${\bm J}\approx \hat{\bm J} = {\bm U}_{s} {\bm \Lambda}_{s} {\bm V}_{s}^*$. This approximation could be formed using a randomized algorithm, such as \cref{alg:gensubit}, but could also be the exact TSVD. For approximately solving \eqref{eq:LMupdate} we consider using the rank-$p$ approximation $\llbracket \hat{\bm J} \rrbracket_p = {\bm U}_{p} {\bm \Lambda}_{p} {\bm V}_{p}^*$, where $p \leq s$, and one of the following approximate LM updates (see also \cref{eq:LMupdatedx1,eq:LMupdatedx2,eq:LMupdatedx3}):
\begin{align}
\delta {\bm x}_1 =
 \sum_{i=1}^p \alpha_i {\bm v}_i
; \quad
\delta {\bm x}_2 = 
 \sum_{i=1}^p \beta_i {\bm v}_i
; \quad
\delta {\bm x}_3 = 
 - \frac{1}{\mu + \gamma} \left[ ({\bm g}_\text{obs} + \mu {\bm x}) + \sum_{i=1}^p \lambda_i^2 \beta_i {\bm v}_i \right]
 ,
\end{align}
where
\begin{align}
\alpha_i = - \frac{1}{\mu + \gamma + \lambda_i^2}  \left[  {\lambda}_i {\bm u}_i^*   {\bm d} +  \mu {\bm v}_i^* {\bm x} \right]
; \quad
\beta_i = - \frac{1}{\mu + \gamma + \lambda_i^2}  \left[  {\bm v}_i^*   {\bm g}_\text{obs} +  \mu {\bm v}_i^* {\bm x} \right]
.
\end{align}

\subsection{Length of parameter update}

For the first two update methods we have
\begin{align}
\norm{\delta {\bm x}_1}^2 = \sum_{i=1}^p \alpha_i^2 
; \quad
\norm{\delta {\bm x}_2}^2 = \sum_{i=1}^p \beta_i^2
.
\end{align}
Therefore, $\norm{\delta {\bm x}_1}$ and $\norm{\delta {\bm x}_1}$ increase monotonically with the truncation $p$, for a given TSVD approximation $\hat{\bm J}$. This is a good property as it means that $p$ can be used to regularize the model update, both in terms of step length as well as filtering out components that would contribute high-frequency components. However, this is not the case for the third updating scheme, as shown below.

For the following discussion we use the matrix ${\bm V} = [{\bm V}_{s} \,\, | \,\, {\bm V}_{\bot}] = [ {\bm v}_1 \,  {\bm v}_2 \, \dots \, {\bm v}_{N_m} ]$, where ${\bm V}_{\bot}$ is a complementary basis orthogonal to ${\bm V}_{s}$ such that ${\bm V}$ defines a complete orthonormal basis for ${\mathbb{R}}^{N_m}$. In that case,
\begin{align}
\delta {\bm x}_3 
= 
\sum_{i=1}^p \beta_i {\bm v}_i 
+ \sum_{i=p+1}^{N_m} \tilde{\beta}_i {\bm v}_i  
,
\end{align}
where
\begin{equation}
\tilde{\beta}_i = - \frac{1}{\mu + \gamma} ( {\bm v}_i^*   {\bm g}_\text{obs} + \mu {\bm v}_i^*   {\bm x}) =  \frac{\mu + \gamma + \lambda_i^2}{\mu + \gamma} \beta_i
,
\end{equation}
that is, $\tilde{\beta}_i = \kappa_i \beta_i$ where $\kappa_i \geq 1$. Accordingly, for $1 \leq p < s$,
\begin{align}
\norm{[\delta {\bm x}_3]_p}^2 
= 
\sum_{i=1}^p \beta_i^2 + \sum_{i=p+1}^{N_m} \tilde{\beta}_i^2
\geq 
\sum_{i=1}^{p+1} \beta_i^2 + \sum_{i=p+2}^{N_m} \tilde{\beta}_i^2
= 
\norm{[\delta {\bm x}_3]_{p+1}}^2
,
\end{align}
since $\tilde{\beta}_i^2 \geq \beta_i^2$. Decreasing $p$, therefore, increases the length of the model update $\delta {\bm x}_3$ and truncating $p$ is not effective as a regularization tool. Furthermore,
\begin{align}
\norm{\delta {\bm x}_3}^2 
= 
\sum_{i=1}^p \beta_i^2 + \sum_{i=p+1}^{N_m} \tilde{\beta}_i^2
= 
\norm{\delta {\bm x}_2}^2  + \sum_{i=p+1}^{N_m} \tilde{\beta}_i^2
\geq \norm{\delta {\bm x}_2}^2 
.
\end{align}
Therefore, when using $\delta {\bm x}_3$, $\gamma$ and $\mu$ become more important for regulating the problem and may need to be chosen larger. A larger $\gamma$ may be needed to make $\norm{\delta {\bm x}_3}$ comparable to $\norm{\delta {\bm x}_2}$. For an exact TSVD or when using a randomized approximation such that $\hat{\bm J} = {\bm J} {\bm Q}_r {\bm Q}_r^*$ then $\delta {\bm x}_1 = \delta {\bm x}_2$ \eqref{eq:dx1dx2same} and we get \cite[Lemma 5.8]{voronin2015lstsq}.

\subsection{Comparison with full update}

We also look at comparing the TSVD updates with the full LM update. Given the exact SVD of ${\bm J}$ whose full rank is $r$, the full LM solution \eqref{eq:LMupdate} can be written as
\begin{align}
	 \delta \bar{\bm x}  =  
     - \left[ {\bm J}^* {\bm J}  + (\mu + \gamma ) {\bm I}\right]^{-1}  ( {\bm J}^* {\bm d}  + \mu {\bm x} )   
= 
\sum_{i=1}^r \beta_i {\bm v}_i 
+
\sum_{i=r+1}^{N_m} \tilde{\beta}_i {\bm v}_i
.
\end{align}
Then using the exact rank-$p$ TSVD the difference between the third update scheme and the full update is
\begin{align}
\delta {\bm x}_3 - \delta \bar{\bm x} = 
\sum_{i=p+1}^{r} (\tilde{\beta_i} - \beta_i) {\bm v}_i
=
\sum_{i=p+1}^{r} \frac{\lambda_i^2}{\mu + \gamma} \beta_i {\bm v}_i
\end{align}
and
\begin{align}
\norm{\delta {\bm x}_3 - \delta \bar{\bm x}} &=
\left\{ \sum_{i=p+1}^{r} \left[ \frac{\lambda_i^2}{(\mu + \gamma)(\mu + \gamma + \lambda_i^2)} \right]^2   
( {\bm v}_i^*   {\bm g}_\text{obs} + \mu {\bm v}_i^*   {\bm x})^2   \right\}^{1/2}
\\
&\leq
\frac{\lambda_{p+1}^2}{(\mu + \gamma)(\mu + \gamma + \lambda_{p+1}^2)}
\norm{{\bm g}_\text{obs} + \mu {\bm x}}
.
\end{align}
For $\mu = 0$ or ${\bm x} = {\bm 0}$ we get the result given by \cite[Prop. 5.7]{voronin2015lstsq}, that is
\begin{align}
\norm{\delta {\bm x}_3 - \delta \bar{\bm x}}_2 &=
\left\{ \sum_{i=p+1}^{r} \left[ \frac{\lambda_i^2}{ (\mu + \gamma) (\mu + \gamma + \lambda_i^2)} \right]^2   
( \lambda_i {\bm u}_i^*  {\bm d})^2   \right\}^{1/2}
\\
&\leq
\frac{\lambda_{p+1}^3}{ (\mu + \gamma) (\mu + \gamma + \lambda_{p+1}^2)}  \norm{ {\bm d} }
.
\end{align}

We can similarly consider $\delta {\bm x}_2$, given the exact TSVD of ${\bm J}$. Then
\begin{align}
\delta {\bm x}_2 - \delta \bar{\bm x} = 
- \sum_{i=p+1}^{r} \beta_i {\bm v}_i
- \sum_{i=r+1}^{N_m} \tilde{\beta}_i {\bm v}_i
\end{align}
and
\begin{align}
& \norm{\delta {\bm x}_2 - \delta \bar{\bm x}}^2 = 
\sum_{i=p+1}^{r} \beta_i^2 
+ \sum_{i=r+1}^{N_m} \tilde{\beta}_i^2 
\\
&=
\sum_{i=p+1}^{r} \left( \frac{1}{\mu + \gamma + \lambda_i^2} \right)^2 ( {\bm v}_i^* [{\bm g}_\text{obs} + \mu {\bm x}] )^2
+ \sum_{i=r+1}^{N_m} \left( \frac{\mu}{\mu + \gamma } \right)^2 ( {\bm v}_i^* {\bm x} )^2 
.
\label{eq:errdx2}
\end{align}
Comparing this with the error for $\delta {\bm x}_3$, the last term in \eqref{eq:errdx2} suggests that $\delta {\bm x}_2$ and $\delta {\bm x}_1$ may have difficulties smoothing out high-frequency components that can build up in ${\bm x}$ during an inversion. This is consistent with our experience, that $\delta {\bm x}_3$ can perform better than $\delta {\bm x}_2$ and $\delta {\bm x}_1$ at lowering the regularization term at late inversion iterations. This is because $\delta {\bm x}_3$ works in the full model parameter space which also helps to lower the overall objective function \eqref{eq:lstsqproblem}. On the other hand, $\delta {\bm x}_2$ and $\delta {\bm x}_1$ are better for regulating model updates at early inversion iterations.

For the case when $\mu=0$ or ${\bm x} = {\bm 0}$ we get
\begin{align}
\norm{\delta {\bm x}_2 - \delta \bar{\bm x}} =
\left\{ \sum_{i=p+1}^{r} \left( \frac{1}{ \mu + \gamma + \lambda_i^2} \right)^2   
( \lambda_i {\bm u}_i^*  {\bm d})^2   \right\}^{1/2}
\leq \max \left( \left\{ \frac{\lambda_i}{ \mu + \gamma + \lambda_i^2}  \right\}_{i=p+1}^r \right)  \norm{ {\bm d} }
,
\end{align}
where
\begin{equation}
\left\{ \frac{\lambda_i}{ \mu + \gamma + \lambda_i^2}  \right\}_{i=p+1}^r = \left[ \frac{\lambda_{p+1}}{ \mu + \gamma + \lambda_{p+1}^2} ,\, \frac{\lambda_{p+2}}{ \mu + \gamma + \lambda_{p+2}^2} ,\, \cdots ,\, \frac{\lambda_{r}}{ \mu + \gamma + \lambda_{r}^2} \right]
.
\end{equation}
For this case \cite[Prop. 5.6]{voronin2015lstsq} gave the bound
\begin{equation} \label{eq:VoroninBound}
\norm{\delta {\bm x}_2 - \delta \bar{\bm x}} 
\leq  \frac{\lambda_{p+1}}{ \mu + \gamma + \lambda_{p+1}^2}   \norm{ {\bm d} }
.
\end{equation}
However, though, \eqref{eq:VoroninBound} holds for $\lambda_{p+1} \leq \sqrt{\mu + \gamma}$ it does not hold in general. Instead if $\lambda_{p+1} > \sqrt{\mu + \gamma}$ and $\lambda_r \leq \sqrt{\mu + \gamma}$, we may expect that
\begin{equation}
\norm{\delta {\bm x}_2 - \delta \bar{\bm x}} 
\leq \max \left( \left\{ \frac{\lambda_i}{ \mu + \gamma + \lambda_i^2}  \right\}_{i=p+1}^r \right)  \norm{ {\bm d} }
\approx  \frac{1}{2 \sqrt{\mu + \gamma}}  \norm{ {\bm d} }
.
\end{equation}
The above approximation for the bound is good when one of the discarded singular values is close to $\sqrt{\mu + \gamma}$. Otherwise, if $\lambda_r > \sqrt{\mu + \gamma}$,
\begin{equation}
\norm{\delta {\bm x}_2 - \delta \bar{\bm x}} 
\leq  \frac{\lambda_{r}}{ \mu + \gamma + \lambda_{r}^2}   \norm{ {\bm d} }
.
\end{equation}

The above error bounds are given for the case of an exact TSVD. Nevertheless, they can also be informative for reasonably good randomized low-rank approximations. The bounds for ${\bm x} = {\bm 0}$ can additionally give insight into what to expect at early inversion iterations when $\mu {\bm x}$ is small compared to the observation gradient ${\bm g}_\text{obs}$.

\section*{Acknowledgments}

The author thanks Farbod (Fred) Roosta-Khorasani for thoughtful discussions on randomized methods, and Prof. Michael J. O'Sullivan, Oliver J. Maclaren, Ruanui Nicholson and N. Benjamin Erichson for helpful discussions and feedback on the manuscript.

\bibliographystyle{siamplain}
\bibliography{references}
\end{document}